\documentclass{amsart}

\usepackage{amsthm}
\usepackage{thmtools}

\usepackage[breaklinks, colorlinks=true, linktocpage]{hyperref} 
\usepackage[capitalize]{cleveref}   
\usepackage{amsmath,amsfonts,amssymb,amscd,comment,euscript}
\usepackage{etoolbox}
\usepackage{stmaryrd}
\usepackage[all]{xy}
\usepackage{graphicx}
\usepackage{enumerate, enumitem} 
\usepackage{pifont}
\usepackage[usenames,dvipsnames]{xcolor}
\usepackage{tikz-cd}
\usepackage{xcolor}
\usepackage{soul}
\usepackage{stmaryrd} 

\renewcommand\rightmark{\footnotesize A technique for computing oriented cohomology rings of semisimple algebraic groups} 

\theoremstyle{plain}
\newtheorem{lem}{Lemma}[section]
\newtheorem{cor}[lem]{Corollary}
\newtheorem{prop}[lem]{Proposition}
\newtheorem{alg}[lem]{Algorithm}
\newtheorem{theo}[lem]{Theorem}

\newtheorem{assumption}[lem]{Assumption}

\theoremstyle{definition}
\newtheorem{egg}[lem]{Example}
\newtheorem{rmk}[lem]{Remark}
\newtheorem{defn}[lem]{Definition}
\newtheorem{nota}[lem]{Notation}

\newcommand{\arxiv}[1]{\href{http://arxiv.org/abs/#1}{\tt arXiv:\nolinkurl{#1}}}

\newcommand{\arxivpdf}[1]{\href{http://arxiv.org/pdf/#1}{\tt arXiv:\nolinkurl{#1}}}

\newcommand{\mrm}{\mathrm}
\newcommand{\mcal}{\mathcal}
\newcommand{\D}{\mathcal{D}}

\newcommand{\Hom}{\mathrm{Hom}}

\newcommand{\aaa}{(\alpha)}
\newcommand{\bbb}{(\beta)}
\newcommand{\ab}{(\alpha,\beta)}
\newcommand{\ba}{(\beta,\alpha)}
\newcommand{\aba}{(\alpha,\beta,\alpha)}
\newcommand{\bab}{(\beta,\alpha,\beta)}
\newcommand{\abab}{(\alpha,\beta,\alpha,\beta)}

\newcommand{\Xa}{\Delta_{\alpha}^*}
\newcommand{\Xb}{\Delta_{\beta}^*}
\newcommand{\Xab}{\Delta_{\ab}^*}
\newcommand{\Xba}{\Delta_{\ba}^*}
\newcommand{\Xaba}{\Delta_{\aba}^*}
\newcommand{\Xbab}{\Delta_{\bab}^*}
\newcommand{\Xabab}{\Delta_{\abab}^*}

\newtoggle{comments}
\newtoggle{details}
\newtoggle{detailsnote}

\toggletrue{comments}   
\togglefalse{details}   

\toggletrue{detailsnote}   

\iftoggle{comments}{%
	\newcommand{\comments}[1]{
		\ \\
		{\color{red}
			\textbf{RG:} #1
		}
		\\
	}
}{%
	\newcommand{\comments}[1]{}
}

\iftoggle{details}{%
	\newcommand{\details}[1]{
		\ \\
		{\color{OliveGreen}
			\textbf{Details:} #1
		}
		\\
	}
}{%
	\newcommand{\details}[1]{}
}

\begin{document}
	
	\title[]{A technique for computing oriented cohomology rings\\ of semisimple algebraic groups}
	
	\author[]{Raj Gandhi}
	
	\address{
		Department of Mathematics  \\
		Cornell University
	}
	\email{rg593@cornell.edu}
	
	\address{} 
	\thanks{}
	
	\address{}
	\thanks{}

	\begin{abstract}
		We present a technique for computing a finite set of generators and relations for the ring $\mrm{h}^*(G)$ in terms of formal Demazure operators, where $\mrm{h}^*$ is an oriented cohomology theory satisfying the localization axiom and $G$ is a semisimple algebraic group. Using this technique, we give minimal presentations for the oriented cohomology rings of the adjoint and simply-connected groups of types $A_1$, $A_2$, and $B_2$.
	\end{abstract}
	
	\maketitle
	
	%
	
	\tableofcontents

	\section{Introduction}

	The Chow ring of a linear algebraic group is a very well-studied invariant. Noteworthy examples of computations of these Chow rings include work of Grothendieck \cite{G58}, who showed that the integral Chow rings of the complex algebraic groups $\mrm{SL}_n$ and $\mrm{Sp}_{2n}$ are isomorphic to $\mathbb{Z}$, and work of Kac \cite{K85}, who developed a procedure for calculating the Chow rings of complex reductive algebraic groups with coefficients in any finite field. The functor that sends a smooth variety to its Chow ring is an example of an \textit{(algebraic) oriented cohomology theory} \cite{LM07}. The general study of oriented cohomology rings of linear algebraic groups is a natural generalization of the study of their Chow rings, though there do not appear to be many computations of general oriented cohomology rings of linear algebraic groups in the literature. With that said, partial computations for the \textit{algebraic cobordism rings} of some simple algebraic groups over $\mathbb{C}$ appear in \cite[\S 5]{Y05}.

	In this paper, we describe a complete finite set of generators and relations for the oriented cohomology rings $\mrm{h}^*(G)$ of all semisimple algebraic groups $G$ over algebraically closed fields of characteristic $0$, assuming that the {oriented cohomology theory} $\mrm{h}^*$ satisfies the \textit{localization axiom} (see \cref{defn:localization}.) The presentation in terms of generators and relations is given in \cref{lem:structure-of-cohomology}. In \cref{alg:only}, we describe a technique for computing the relations in the ring $\mrm{h}^*(G)$ explicitly. In \cref{section:Examples-in-rank-2}, we apply this technique to explicitly compute all relations in $\mrm{h}^*(G)$, where $G$ is an adjoint or simply-connected group of type $A_1$, $A_2$, or $B_2$; for these groups, we simplify the relations of \cref{lem:structure-of-cohomology} by hand to obtain minimal presentations for the oriented cohomology rings of these adjoint and simply-connected groups, which are displayed in \cref{tab:title}. We believe that the techniques used in this paper can be adapted to compute minimal presentations for the oriented cohomology rings of other semisimple algebraic groups. We will now give a detailed summary of the techniques used in this paper.

	Let $\mrm{h}^*$ be an oriented cohomology theory in the sense of Levine-Morel \cite{LM07}. This is a contravariant functor from the category of smooth quasi-projective varieties over a field $k$ to the category of commutative, graded rings satisfying certain axioms. Examples of oriented cohomology theories include the Chow theory $\mrm{CH}^*$, Grothendieck's $K^0$ functor\footnote{The notation $K_0$ is standard in the literature. We follow the notation used by Levine--Morel (\cite[Example 1.1.5]{LM07}).}, and, over a base field of characteristic $0$, the universal oriented cohomology theory called algebraic cobordism $\Omega^*$.	Given a vector bundle $E\to X$ of rank $n$ on a smooth variety $X$ over $k$, there is a set of \textit{Chern classes} $c_i(E)\in \mrm{h}^i(X)$, $i\in\{0,\dotsc,n\}$. The Chern classes satisfy Quillen's formula: given two line bundles $\mcal{L}_1$ and $\mcal{L}_2$ on $X$, we have
	\[c_1(\mcal{L}_1\otimes \mcal{L}_2)=F(c_1(\mcal{L}_1),c_1(\mcal{L}_2)),\]
	where $F$ is a one-dimensional commutative formal group law over the ring $R=\mrm{h}^*(\mrm{Spec}(k))$.

	Let $\mrm{h}^*$ be an oriented cohomology theory  with formal group law $F$ over $R=\mrm{h}^*(\mrm{Spec}(k))$ satisfying the \textit{localization axiom} (Assumption \ref{assumption:restrict-theory}), and let $G$ be a semisimple algebraic group over an algebraically closed field $k$ of characteristic $0$ with root datum $\Psi=(\Sigma,\Lambda,\Sigma^\vee,\Lambda^\vee)$. Two key tools used in this paper are the \textit{formal group algebra} $R\llbracket\Lambda\rrbracket_F$ of \cite{CPZ}, which can be viewed as an algebraic substitute for the torus-equivariant oriented cohomology ring of a point (\cite{CZZ3}), and the  subalgebra $\mcal{D}_F$ of the algebra of $R$-linear endomorphisms of $R\llbracket\Lambda\rrbracket_F$ generated by \textit{formal Demazure operators} and by multiplication by elements in $R\llbracket\Lambda\rrbracket_F$. We will always assume that $R\llbracket\Lambda\rrbracket_F$ satisfies a technical assumption, Assumption \ref{assumption:sigma-regular}, which is crucial for the results in Sections \ref{section:Formal-affine-Demazure-algebras} and \ref{section:Algebra-of-Demazure-operators} to hold.
	Fix a Borel subgroup $B$ of $G$, so that $G/B$ is a complete flag variety. In \cite[Theorem 5.1]{GZ12}, Gille-Zainoulline derived a relationship between the oriented cohomology rings $\mathrm{h}^*(G)$ and $\mathrm{h}^*(G/B)$ through the first Chern class map $\mrm{c}_1$. The precise relationship is 
	\begin{equation}\label{GZ}\mrm{h}^*(G)\simeq \mrm{h}^*(G/B)/\left(\mrm{im}\left(\mrm{c}_1|_{\mcal{L}_B}\right)\right),
	\end{equation}
	where $\mrm{c}_1|_{\mcal{L}_B}$ is the restriction of $\mrm{c}_1$ to the set $\mcal{L}_B$ of \textit{$B$-equivariant} line bundles on $G/B$.
	In \cite{CPZ}, Calm\`es-Petrov-Zainoulline constructed an algebraic model for $\mrm{h}^*(G/B)$ using the \textit{augmentation} $\epsilon\D_F$ of the algebra $\D_F$ containing the formal Demazure operators (\cref{section:Structure-of-oriented-cohomology-of-algebraic-groups}). That is, they showed there is an $R$-algebra isomorphism $\mrm{h}^*(G/B)\simeq (\epsilon\D_F)^*$, where $(\epsilon\D_F)^*:=\mrm{Hom}_R(\epsilon\D_F,R)$ is the $R$-dual of $\epsilon \D_F$ (the $R$-algebra structure on $(\epsilon\D_F)^*$ is induced by an $R$-coalgebra structure on $\epsilon \D_F$ (\cite[\S 11]{CZZ1})). The image of $\mrm{c}_1|_{\mcal{L}_B}$ in $\mrm{h}^*(G/B)$ corresponds to the image in the dual $(\epsilon\D_F)^*$ of an \textit{algebraic} Chern class map. Using this algebraic Chern class map, together with (\ref{GZ}), we obtain an algebraic model for $\mrm{h}^*(G)$ in \cref{prop:cohomology-model}. We then describe this algebraic model in terms of a finite set of generators and relations in \cref{lem:structure-of-cohomology}, and provide a technique for computing the relations explicitly in \cref{alg:only}. With enough computing power, a computer can output the list of generators and relations in the ring $\mrm{h}^*(G)$ explicitly, and, from these relations, one can compute a minimal presentation for $\mrm{h}^*(G)$ in terms of generators and relations by hand. In Section \ref{section:Examples-in-rank-2}, we perform this calculation for the adjoint and simply-connected algebraic groups of types $A_1$, $A_2$, and $B_2$.

	This paper is organized as follows. In Section \ref{section:Root-datum}, we recall the notion of a root datum. In Section \ref{section:Formal-affine-Demazure-algebras}, we recall the formal group algebra and the formal affine Demazure algebra. In Section \ref{section:Algebra-of-Demazure-operators}, we recall the algebra and coalgebra structures on $\mcal{D}_F$ and discuss the algebra structure on the dual $\mcal{D}_F^*$. In \cref{section:line-bundles}, we prove a key result about homogeneous line bundles on complete flag varieties. In Section \ref{section:oriented-cohomology-theories}, we recall the notion of an oriented cohomology theory. In Section \ref{section:Structure-of-oriented-cohomology-of-algebraic-groups}, we provide a presentation for the oriented cohomology ring $\mrm{h}^*(G)$ in terms of generators and relations. In Section \ref{section:Examples-in-rank-2}, we compute a minimal presentation for $\mrm{h}^*(G)$ in terms of generators and relations, where $G$ is an adjoint or simply-connected algebraic group of type $A_1$, $A_2$, or $B_2$.

	\textbf{Acknowledgements}: This paper contains results from the author's master's thesis \cite{Gan21}. The author would like to thank Kirill Zainoulline for his guidance throughout this project. The author would also like to thank Burt Totaro for his helpful comments on this article, Ruizhen Liu for his help with the proof of \cref{lem:isomorphic-line-bundles0}, and Gabe Udell for helping verify several calculations in \cref{section:Examples-in-rank-2}. This project was partially supported by an Undergraduate Student Research Award and a Canada Graduate Scholarship - M.Sc. - from the Natural Sciences and Engineering Research Council of Canada. It was also supported by an Ontario Graduate Scholarship and by the University of Ottawa.

	\section{Root datum}\label{section:Root-datum}
	In this section, we recall the notion of a root datum. We closely follow \cite[\S2 and \S 4]{CZZ1}.
	
	\begin{defn}{\textup{(see} \cite[Exp. XXI, \S1.1]{SGA} and \cite[\S2 and \S 4]{CZZ1}.\textup{)}} A \textit{root datum} is a quadruple $\Psi=(\Sigma,\Lambda,\Sigma^\vee,\Lambda^\vee)$, where $\Lambda$ is a lattice (i.e., a finitely generated free abelian group), $\Lambda^\vee$ is the lattice dual to $\Lambda$, and $\Sigma$ is a non-empty finite subset of $\Lambda$ equipped with an embedding $\Sigma\hookrightarrow \Lambda^\vee$, $\alpha\mapsto\alpha^\vee$, with image $\Sigma^\vee$, satisfying both:
		\begin{enumerate}
			\item $\Sigma\cap 2\Sigma=\emptyset$, and $\alpha^\vee(\alpha)=2$ for all $\alpha\in \Sigma$;
			\item For every $\alpha\in \Sigma$ there are $\mathbb{Z}$-linear automorphisms $s_
			\alpha$ of $\Lambda$ and $s_\alpha^\vee$ of $\Lambda^\vee$, given by the following formulas:
			\[s_\alpha(x)=x-\alpha^\vee(x)\alpha,\quad s_\alpha^\vee(y)=y-y(\alpha)\alpha^\vee,\]
			for all $x\in \Lambda$, $y\in\Lambda^\vee$. Moreover, $s_\alpha(\Sigma)=\Sigma$ and $s_\alpha^\vee(\Sigma^\vee)=\Sigma^\vee$ for all $\alpha\in\Sigma$.
		\end{enumerate}
	\end{defn}

	Fix a root datum $\Psi=(\Sigma,\Lambda,\Sigma^\vee,\Lambda^\vee)$. The set $\Sigma$ is a \textit{root system}. Its elements are called \textit{roots}, and the sublattice $\Lambda_r$ of $\Lambda$ generated by $\Sigma$ is called the \textit{root lattice}. There is always a basis $\Delta=\{\alpha_1,\dotsc,\alpha_n\}$ for the root lattice $\Lambda_r$ called a \textit{simple system}, such that each $\alpha\in\Sigma$ is a linear combination of the $\alpha_i$ with all coefficients positive (in which case $\alpha$ is a \textit{positive root}) or all coefficients negative (in which case $\alpha$ is a \textit{negative root}). Thus, $\Sigma=\Sigma_+\sqcup\Sigma_-$ is the disjoint union of positive roots $\Sigma_+$ and negative roots $\Sigma_-$. The $\alpha_i$ are \textit{simple roots}. Set $\Lambda_\mathbb{Q}:=\Lambda\otimes\mathbb{Q}$. The root datum is \textit{semisimple} if $\Lambda_\mathbb{Q}=\Lambda_r\otimes\mathbb{Q}$. Henceforth, we will only consider semisimple root data.

	The \textit{rank} of the semisimple root datum is the rank of $\Lambda_r$. The root datum is \textit{irreducible} if it is not the disjoint union of root data of smaller ranks. The automorphisms $s_\alpha$ are called \textit{reflections}, and the subgroup of $\mathbb{Z}$-linear automorphisms of $\Lambda$ generated by $\{s_\alpha\mid \alpha\in\Sigma\}$ is the \textit{Weyl group} $W$ of the root datum. In fact, $W$ is generated by $\{s_{\alpha_i}\mid \alpha_i\in\Delta\}$.
	
	The sublattice of $\Lambda_\mathbb{Q}=\Lambda\otimes \mathbb{Q}$ generated by elements $\lambda$ satisfying $\alpha^\vee(\lambda)\in\mathbb{Z}$ for all $\alpha\in\Sigma$ is the \textit{weight lattice} $\Lambda_w$. The weight lattice always has a basis consisting of \textit{fundamental weights} $\{\lambda_1,\dotsc,\lambda_n\}$ with respect to $\Delta$: the fundamental weights are defined by the property $\alpha_i^\vee(\lambda_j)=\delta_{i,j}$, where $\delta_{i,j}$ is the Kronecker function. If $\Lambda=\Lambda_w$, then the root datum is called \textit{simply-connected}. If $\Lambda=\Lambda_r$, the root datum is called \textit{adjoint}. We call $\Lambda_w/\Lambda$ the \textit{fundamental group} of the root datum. We denote an irreducible simply-connected (resp. adjoint) root datum by $\mathcal{D}_n^{\text{sc}}$ (resp. $\mathcal{D}_n^{\text{ad}}$), where $\mathcal{D}$ is one of the Dynkin types, $A,B,C,D,E,F,G$, and $n$ is the rank of the root datum. An irreducible root datum is entirely determined by its Dynkin type $\mcal{D}$ and the lattice $\Lambda$. The group $\Lambda_w/\Lambda_r$ has the following structure, depending on the Dynkin type:
	
	\begin{center}
		\begin{tabular}{ |c |c |}
			\hline
			Root system & $\Lambda_w/\Lambda_r$ \\
			\hline\hline
			$A_n$ & $\mathbb{Z}/(n+1)\mathbb{Z}$\\
			\hline
			$B_n, C_n,E_7$ & $\mathbb{Z}/2\mathbb{Z}$ \\
			\hline
			$D_n$ & $\left\{\begin{aligned}
				\mathbb{Z}/2\mathbb{Z}\times \mathbb{Z}/2\mathbb{Z},\quad\text{$n$ even},\\
				\mathbb{Z}/4\mathbb{Z},\quad\quad\quad\quad\textcolor{white}{.,}\text{$n$ odd}.
			\end{aligned}\right\}$ \\
			\hline
			$E_6$ & $\mathbb{Z}/3\mathbb{Z}$ \\
			\hline
			$E_8,F_4,G_2$ & $0$\\
			\hline
		\end{tabular}
	\end{center}

	We can write any root $\alpha\in\Sigma$ as a linear combination $\sum_{t=1}^n d_{\alpha,t} \lambda_t$ for some $d_{\alpha,t}\in\mathbb{Z}$. If $\alpha=\alpha_i$ is a simple root, then we have
	\[C_{j,i}=\alpha_j^\vee(\alpha_i)=\sum_{t=1}^nd_{\alpha_i,t} \alpha_j^\vee(\lambda_t)=d_{\alpha_i,j},\]
	where $C_{j,i}$ is the $(i,j)$-th entry of the \textit{Cartan matrix} of $\Sigma$.
	Thus, $\alpha_i=\sum_{j=1}^nC_{j,i}\lambda_j$. Moreover, one can solve the system of linear equations $\{\alpha_i=\sum_{j=1}^nC_{j,i}\lambda_j\}$ in order to express $\lambda_j$ as a $\mathbb{Q}$-linear combination of the $\alpha_i$. If $\Sigma$ has rank $1$, then $A_1$ is the unique root system. If $\Sigma$ has rank $2$, there are exactly four root systems: $A_1\times A_1$, $A_2$, $B_2$, and $G_2$. Below, we compute the fundamental weights for the rank $2$ root systems in terms of the simple roots. These will be used for the computations in Section \ref{section:Examples-in-rank-2}.
	
	\begin{egg}\label{egg:A1xA1}
		Let $\Sigma$ be the root system of type $A_1\times A_1$. The Cartan matrix of $\Sigma$ is
		\[\begin{bmatrix}
			2& 0\\
			0& 2
		\end{bmatrix}.\]
		Thus, $\alpha_1=2\lambda_1$ and $\alpha_2=2\lambda_2$. We can write 
		\[\lambda_1=\tfrac{1}{2}\alpha_1\quad\text{and}\quad\lambda_2=\tfrac{1}{2}\alpha_2.\] 
		One can directly verify that
		\[\alpha_1^\vee(\lambda_1)=1;\quad \alpha_1^\vee(\lambda_2)=0;\quad \alpha_2^\vee(\lambda_1)=0;\quad \alpha_2^\vee(\lambda_2)=1.\]
		Note also that
		\[s_2(\alpha_1)=\alpha_1;\quad s_1(\alpha_1)=-\alpha_1;\quad s_2s_1(\alpha_1)=-\alpha_1.\]
		\[s_2(\alpha_2)=-\alpha_2;\quad s_1(\alpha_2)=\alpha_2;\quad s_2s_1(\alpha_2)=-\alpha_2.\]
	\end{egg}
	
	\begin{egg}\label{egg:A2}
		The Cartan matrix of $A_2$ is:
		\[\begin{bmatrix}
			2& -1\\
			-1& 2
		\end{bmatrix}.\]
		Thus, $\alpha_1=2\lambda_1-\lambda_2$ and $\alpha_2=-\lambda_1+2\lambda_2$. We can write 
		\[\lambda_1=\tfrac{2}{3}\alpha_1+\tfrac{1}{3}\alpha_2\quad\text{and}\quad\lambda_2=\tfrac{1}{3}\alpha_1+\tfrac{2}{3}\alpha_2. \]
		One can directly verify that
		\[\alpha_1^\vee(\lambda_1)=1;\quad \alpha_1^\vee(\lambda_2)=0;\quad \alpha_2^\vee(\lambda_1)=0;\quad \alpha_2^\vee(\lambda_2)=1.\]
		Note also that
		\[s_1(\alpha_1)=-\alpha_1;\quad s_2(\alpha_1)=\alpha_1+\alpha_2;\quad s_1s_2(\alpha_1)=\alpha_2;\]\[ s_2s_1(\alpha_1)=-\alpha_1-\alpha_2;\quad s_1s_2s_1(\alpha_1)=-\alpha_2.\]
		\[s_1(\alpha_2)=\alpha_1+\alpha_2;\quad s_2(\alpha_2)=-\alpha_2;\quad s_1s_2(\alpha_2)=-\alpha_1-\alpha_2;\]\[ s_2s_1(\alpha_2)=\alpha_1;\quad s_1s_2s_1(\alpha_2)=-\alpha_1.\]
	\end{egg}
	
	\begin{egg}\label{egg:B2}
		The Cartan matrix of $B_2$ is:
		\[\begin{bmatrix}
			2& -2\\
			-1& 2
		\end{bmatrix}.\]
		Thus, $\alpha_1=2\lambda_1-2\lambda_2$ and $\alpha_2=-\lambda_1+2\lambda_2$. We can write 
		\[\lambda_1=\alpha_1+\alpha_2\quad\text{and}\quad\lambda_2=\tfrac{1}{2}\alpha_1+\alpha_2.\] 
		One can directly verify that
		\[\alpha_1^\vee(\lambda_1)=1;\quad \alpha_1^\vee(\lambda_2)=0;\quad \alpha_2^\vee(\lambda_1)=0;\quad \alpha_2^\vee(\lambda_2)=1.\]	
		Note also that 
		\[s_2(\alpha_1)=2\alpha_2+\alpha_1;\quad s_1(\alpha_1)=-\alpha_1;\quad s_2s_1(\alpha_1)=-2\alpha_2-\alpha_1;\quad s_1s_2(\alpha_1)=2\alpha_2+\alpha_1;\]
		\[s_2s_1s_2(\alpha_1)=\alpha_1;\quad s_1s_2s_1(\alpha_1)=-2\alpha_2-\alpha_1;\quad s_2s_1s_2s_1(\alpha_1)=-\alpha_1.\]
		\[s_2(\alpha_2)=-\alpha_2;\quad s_1(\alpha_2)=\alpha_2+\alpha_1;\quad s_2s_1(\alpha_2)=\alpha_2+\alpha_1;\quad s_1s_2(\alpha_2)=-\alpha_2-\alpha_1;\]
		\[s_2s_1s_2(\alpha_2)=-\alpha_2-\alpha_1;\quad s_1s_2s_1(\alpha_2)=\alpha_2;\quad s_2s_1s_2s_1(\alpha_2)=-\alpha_2.\]
	\end{egg}

	\begin{egg}\label{egg:G2}
		The Cartan matrix of $G_2$ is:
		\[\begin{bmatrix}
			2& -1\\
			-3& 2
		\end{bmatrix}.\]
		Thus, $\alpha_1=2\lambda_1-\lambda_2$ and $\alpha_2=-3\lambda_1+2\lambda_2$. We can write 
		\[\lambda_1=2\alpha_1+\alpha_2\quad\text{and}\quad\lambda_2=3\alpha_1+2\alpha_2.\] 
		Since the $\lambda_i$ are expressible as a $\mathbb{Z}$-linear combination of $\alpha_1$ and $\alpha_2$, the weight lattice and root lattice of $G_2$ coincide. We have
		\[s_2(\alpha_1)=\alpha_2+\alpha_1;\quad s_1(\alpha_1)=-\alpha_1;\quad s_2s_1(\alpha_1)=-\alpha_2-\alpha_1;\quad s_1s_2(\alpha_1)=\alpha_2+2\alpha_1;\] \[s_2s_1s_2(\alpha_1)=\alpha_2+2\alpha_1;\quad s_1s_2s_1(\alpha_1)=-\alpha_2-2\alpha_1;\quad s_2s_1s_2s_1(\alpha_1)=-\alpha_2-2\alpha_1;\]
		\[s_1s_2s_1s_2(\alpha_1)=\alpha_2+\alpha_1;\quad s_2s_1s_2s_1s_2(\alpha_1)=\alpha_1;\quad s_1s_2s_1s_2s_1(\alpha_1)=-\alpha_2-\alpha_1;\]\[s_2s_1s_2s_1s_2s_1(\alpha_1)=-\alpha_1.\]
		\[s_2(\alpha_2)=-\alpha_2;\quad s_1(\alpha_2)=\alpha_2+3\alpha_1;\quad s_2s_1(\alpha_2)=2\alpha_2+3\alpha_1;\quad s_1s_2(\alpha_2)=-\alpha_2-3\alpha_1;\]
		\[s_2s_1s_2(\alpha_2)=-2\alpha_2-3\alpha_1;\quad  s_1s_2s_1(\alpha_2)=2\alpha_2+3\alpha_1;\quad s_2s_1s_2s_1(\alpha_2)=\alpha_2+3\alpha_1;\]
		\[s_1s_2s_1s_2(\alpha_2)=-2\alpha_2-3\alpha_1;\quad s_2s_1s_2s_1s_2(\alpha_2)=-\alpha_2-3\alpha_1;\quad s_1s_2s_1s_2s_1(\alpha_2)=\alpha_2;\]
		\[s_2s_1s_2s_1s_2s_1(\alpha_2)=-\alpha_2.\]
	\end{egg}

	\section{Formal affine Demazure algebra}\label{section:Formal-affine-Demazure-algebras}
	
	In this section, we recall the definitions of the formal group algebra and the formal affine Demazure algebra. We closely follow \cite{CPZ}, \cite{CZZ1}, and \cite{HMSZ}. 
	
	\begin{defn}{\textup{(see} \cite[pp. 4]{LM07}.\textup{)}}
		A \textit{one-dimensional commutative formal group law} $F$ over a commutative ring $R$ is a power series $F(u,v)\in R\llbracket u,v\rrbracket$ satisfying the following axioms:
		\begin{enumerate}
			\item\label{itm:one} $F(u,0)=F(0,u)=u\in R\llbracket u\rrbracket$;
			\item\label{itm:two} $F(u,v)=F(v,u)$;
			\item\label{itm:three} $F(u,F(v,w))=F(F(u,v),w)\in R\llbracket u,v,w\rrbracket$.
		\end{enumerate}
		The \textit{formal inverse} of $u$ is the unique power series in $R\llbracket u\rrbracket$, denoted by $-_Fu$, such that $F(u,-_Fu)=0$
	\end{defn}
	
	Let $F=F(u,v)\in R\llbracket u,v\rrbracket$ be a one-dimensional commutative formal group law over a commutative ring $R$. For $m\in\mathbb{Z}_{\geq 0}$, we use the notation
	\[u+_Fv:=F(u,v),\quad m\cdot_F u:=\underbrace{u+_F\dotsb+_Fu}_{\text{$m$-times}},\quad \text{and}\quad (-m)\cdot_F u:=-_F(m\cdot_F u).\]
	Write $F(u,v)=\sum_{i,j\geq 0}a_{i,j}u^iv^j$, where $a_{i,j}\in R$.
	The coefficients $a_{i,j}$ have relations imposed on them by the axioms of the formal group law. For example, axiom (\ref{itm:one}) implies that $a_{0,0}=0$ and that $a_{1,0}=a_{0,1}=1$, and axiom (\ref{itm:two}) implies that $a_{i,j}=a_{j,i}$ for all $i,j\geq 0$. The relations imposed by axiom (\ref{itm:three}) are complicated.
	
	We can also write $-_Fu=-\sum_{i\geq 1}c_iu^i$ for some coefficients $c_i\in R$. The coefficients $c_i$ can be expressed in terms of the coefficients $a_{i,j}$ through the equation $F(u,-_Fu)=0$, as illustrated in the example below.
	\begin{egg}\label{egg:FGL-inverse}
		We will show how $c_1$, $c_2$, $c_3$, $c_4$, $c_5$, and $c_6$ can be expressed in terms of the $a_{i,j}$:
		\begin{align*}
			c_1&=1;\\
			c_2&=-a_{1,1};\\
			c_3&=-a_{1,1}c_2;\\
			c_4&=-a_{1,1}c_3+a_{1,2}c_2-2a_{1,3}+a_{2,2};\\
			c_5&=-a_{1,1}c_4+a_{1,2}c_3+a_{1,2}c_2^2-4a_{1,3}c_2+2a_{2,2}c_2;\\
			c_6&=-a_{1,1}c_5+a_{1,2}c_4+2a_{1,2}c_2c_3-4a_{1,3}c_3-3a_{1,3}c_2^2+3a_{1,4}c_2-2a_{1,5}+2a_{2,2}c_3\\&\textcolor{white}{....}+a_{2,2}c_2^2-a_{2,3}c_2+2a_{2,4}-a_{3,3}.
		\end{align*} 
	\end{egg}

	Fix a root datum $\Psi=(\Sigma,\Lambda,\Sigma^\vee,\Lambda^\vee)$, with Weyl group $W$. Let $R[x_\Lambda]$ be the polynomial ring over $R$ with variables indexed by $\Lambda$. The augmentation map $\epsilon': R[x_\Lambda]\to R$ sends $x_\lambda$ to $0$ for each $\lambda\in\Lambda$ and fixes $R$. Let $R\llbracket x_\Lambda\rrbracket$ be the $\mrm{ker}(\epsilon')$-adic completion of the polynomial ring $R[x_\Lambda]$, and let $\mcal{J}_F$ be the closure of the ideal in $R\llbracket x_\Lambda\rrbracket$ generated by $x_0$ and elements of the form $x_{\lambda_1+\lambda_2}-(x_{\lambda_1}+_Fx_{\lambda_2})$ for all $\lambda_1,\lambda_2\in\Lambda$.
	\begin{defn}(\cite[Def. 2.4]{CPZ})
		The \emph{formal group algebra} is the quotient
		\[R\llbracket\Lambda\rrbracket_F:=R\llbracket x_\Lambda\rrbracket/\mcal{J}_F.\]
	\end{defn}

	The formal group algebra $R\llbracket\Lambda\rrbracket_F$ is a complete Hausdorff ring with respect to the $\mcal{I}_F$-adic topology, where $\mcal{I}_F$ is the kernel of the augmentation map $\epsilon\colon R\llbracket\Lambda\rrbracket_F\to R$ that sends $x_\lambda\mapsto 0$ for each $\lambda\in\Lambda$ and fixes $R$. By \cite[Cor. 2.13]{CPZ}, if $\{\lambda_i\}$ is a $\mathbb{Z}$-basis of $\Lambda$, then there is an $R$-algebra isomorphism $R\llbracket\Lambda\rrbracket_F\simeq R\llbracket x_1,\dotsc,x_n\rrbracket$,
	\[x_{\sum m_i\lambda_i}\mapsto m_1\cdot_F x_1+_F\dotsb+_F m_n\cdot_F x_n.\]
	There is an action of the Weyl group $W$ on $R\llbracket\Lambda\rrbracket_F$, where $w\in W$ acts by
	\[w(x_\lambda):=x_{w(\lambda)},\quad \lambda\in\Lambda.\]
	Observe that, for each element $\lambda\in\Lambda$, we have $x_{-\lambda}=-_F x_\lambda$.

	\begin{egg}(\cite[\S 2]{CPZ})\label{ex:FGL}
		\begin{enumerate}[label=(\alph*)]
			\item\label{ex:additive:FGL} Let $F_A$ be the \textit{additive} formal group law $F_A(x,y)=x+y$ over $R$. There is an $R$-algebra isomorphism
			\[R\llbracket \Lambda\rrbracket_{F_A}\xrightarrow{\simeq} S^*_R(\Lambda):=\prod_{i=0}^\infty S_R^i,\quad x_\lambda\mapsto \lambda\in S_R^1(\Lambda)\text{ for all $\lambda\in\Lambda$},\]
			where $S_R^i$ is the $i$-th symmetric power of $\Lambda$ over $R$.
			\item\label{ex:mult:FGL} Let $F_M$ be the \textit{multiplicative-periodic} formal group law $F_M(x,y)=x+y-\beta xy$ over $R$, where $\beta\in R^\times$ is a unit. Consider the group ring
			\[R[\Lambda] := \left\{\sum_j r_je^{\lambda_j} \text{ }|\text{ } r_j\in R\text{ and } \lambda_j\in\Lambda\right\}.\]
			Let $\mrm{tr}:R[\Lambda]\to R$ be the trace map, which sends $e^{\lambda}\mapsto 1$ for all $\lambda\in\Lambda$ and fixes $R$. We denote the  $\mrm{ker}(\mrm{tr})$-adic completion of $R[\Lambda]$ by $R[\Lambda]^\wedge$. There is an $R$-algebra isomorphism
			\[R\llbracket \Lambda\rrbracket_{F_M}\xrightarrow{\simeq} R[\Lambda]^\wedge,\]
			induced by $x_\lambda\mapsto \beta^{-1}(1-e^{-\lambda})$ for all $\lambda\in\Lambda$, with inverse $e^{\lambda}\mapsto (1-\beta x_{-\lambda})$.
			\item\label{itm:U-FGL} Let $\mathbb{L}$ be the \textit{Lazard ring}, i.e., the commutative ring with generators $a_{i,j}$, $i,j\geq 1$, subject only to the relations imposed by the axioms of the formal group law. The \textit{universal} formal group law is the formal power series
			\[F_U(u,v)=u+v+\sum_{i,j\geq 1}a_{i,j}u^iv^j\in\mathbb{L}\llbracket u,v\rrbracket.\]
			For any one-dimensional commutative formal group law $F$ over a commutative unital ring $R$, there is a unique ring homomorphism $f\colon\mathbb{L}\to R$ such that $F(u,v)=u+v+\sum_{i,j\geq 1}f(a_{i,j})u^iv^j$. 
		\end{enumerate}
	\end{egg}
	
	We say an element $r\in R\llbracket\Lambda\rrbracket_F$ is \textit{regular} if it is not a zero divisor in $R\llbracket\Lambda\rrbracket_F$. 
	\begin{defn}(\cite[Def. 4.4]{CZZ1})
		We say that the formal group algebra $R\llbracket\Lambda\rrbracket_F$ is \emph{$\Sigma$-regular} if $x_\alpha$ is regular in $R\llbracket\Lambda\rrbracket_F$ for all $\alpha\in\Sigma$. 
	\end{defn}
	\begin{rmk}
		By \cite[Remark 4.5]{CZZ1}, the ring $R\llbracket\Lambda\rrbracket_F$ is $\Sigma$-regular if $2$ is regular in $R$, or if the root datum does not have an irreducible component of type $C_n^{\text{sc}}$, $n\geq 1$. Additionally, since all nonzero elements in an integral domain are regular, it follows that $R\llbracket\Lambda\rrbracket_F$ is $\Sigma$-regular if $R\llbracket\Lambda\rrbracket_F$ is an integral domain, which is true if $R$ is an integral domain (\cite[Corollary 2.13]{CPZ}).
	\end{rmk}
	Let $\mathfrak{t}\in\mathbb{Z}$ be the \textit{torsion index} of the root datum $\Psi$, as introduced in \cite{D73}. The prime divisors of $\mathfrak{t}$ are the \textit{torsion primes} of the corresponding simply-connected root datum, together with the prime divisors of $|\Lambda_w/\Lambda|$. We copy the table of torsion primes and prime divisors of $|\Lambda_w/\Lambda_r|$ from \cite[\S 2]{CZZ1}:
	
	\begin{center}
		\begin{tabular}{ |c ||c |c |c |c |c |c |c |c |c |c | }
			\hline
			Root system & $A_l$ & $B_l$, $l\geq 3$ & $C_l$ & $D_l$, $l\geq 3$ & $G_2$ & $F_4$ & $E_6$ & $E_7$ & $E_8$ \\
			\hline
			$|\Lambda_w/\Lambda_r|$ & $l+1$ & $2$ & $2$ & $4$ & $1$ & $1$ & $3$ & $2$ & $1$ \\
			\hline
			Torsion primes & $\emptyset$ & $2$ & $\emptyset$ & $2$ & $2$ & $2,3$ & $2,3$ & $2,3$ & $2,3,5$ \\
			\hline
		\end{tabular}
	\end{center} 
	
	\begin{rmk}
		The torsion index has been computed explicitly for all simply-connected root data (see \cite{D73}, \cite{T05}, and \cite{T05p2}).
	\end{rmk}

	For the remainder of this paper, we work under the following assumption. This assumption is used for \cref{lem:divisible}, \cref{prop:basis-elements}, \cref{theo:presentation}, and \cref{theo:iso}, and \cref{prop:cocommutative-coproduct}.
	\begin{assumption}{\label{assumption:sigma-regular}}
		We assume $R\llbracket\Lambda\rrbracket_F$ is $\Sigma$-regular, and that $\mathfrak{t}$ is regular in $R$.
	\end{assumption}

	The following lemma will allow us to define formal Demazure operators in \cref{defn:formal-Demazure-operator}.
	
	\begin{lem}(\cite[Lemma 4.1]{CZZ1})\label{lem:divisible}
		For each $u\in R\llbracket\Lambda\rrbracket_F$ and root $\alpha$, the element $u-s_\alpha(u)$ is uniquely divisible by $x_\alpha$.
	\end{lem}
	
	\begin{defn}\label{defn:formal-Demazure-operator}(\cite[Def. 4.2]{CZZ1})
		For each root $\alpha\in\Sigma$, define the \textit{formal Demazure operator}:
		\[\Delta_\alpha: R\llbracket\Lambda\rrbracket_F\to R\llbracket\Lambda\rrbracket_F,\quad\quad \Delta_\alpha(u):=\frac{u-s_\alpha(u)}{x_\alpha},\quad u\in R\llbracket\Lambda\rrbracket_F.\]
	\end{defn}
	\begin{rmk}
		The formal Demazure operator satisfies a \textit{twisted Leibnitz rule}. For all $u,v\in R\llbracket\Lambda\rrbracket_F$, we have
		\[\Delta_\alpha(uv)=\Delta_\alpha(u)v+s_\alpha(u)\Delta_\alpha(v),\quad \alpha\in\Sigma.\]
	\end{rmk}

	\begin{defn}(\cite[\S 5]{CZZ1})
		Let $\mcal{Q}^{F}$ be the localization of $R\llbracket\Lambda\rrbracket_F$ at the multiplicative set generated by the regular elements $\{x_\alpha\mid \alpha\in\Sigma\}$. The action of $W$ on $R\llbracket\Lambda\rrbracket_F$ induces an action by automorphisms on $\mcal{Q}^F$. The \emph{twisted formal group algebra} $\mcal{Q}_W^F$ is the tensor product $\mcal{Q}^F\otimes_R R[W]$ as an $R$-module, with multiplication 
		\[(q\delta_{w})\cdot (q'\delta_{w'}):=qw(q')\delta_{ww'},\quad q,q'\in\mcal{Q}^F,w,w'\in W,\]
		where, $\delta_w$ is the element in the group algebra $R[W]$ corresponding to $w\in W$.
	\end{defn}
	
	\begin{defn}(\cite[Def. 5.3]{CZZ1})
		For each root $\alpha\in\Sigma$, we define the \emph{formal Demazure element},
		\[X_\alpha:=\tfrac{1}{x_\alpha}(1-\delta_{s_\alpha})\in\mcal{Q}_W^F.\]	
	\end{defn}
	
	\begin{defn}(\cite[Def. 5.7]{CZZ1})
		The \emph{formal affine Demazure algebra} $\mathbf{D}_F$ is the $R$-subalgebra of $\mcal{Q}_W^F$ generated by $R\llbracket\Lambda\rrbracket_F$ and by the formal Demazure elements $X_\alpha$, $\alpha\in\Sigma$.
	\end{defn}
	
	For $i>0$, define $[i]:=\{1,\dotsc,i\}$. Let $\Delta=\{\alpha_1,\dotsc,\alpha_n\}$ be a simple system for $\Sigma$, with corresponding simple reflections $\{s_1,\dotsc,s_n\}$. Suppose $m_{i,j}$ is the order of $s_is_j$ in $W$. 
	For $i\in[n]$, we set $\Delta_i:=\Delta_{\alpha_i}$ and $X_i:=X_{\alpha_i}$. If $I=(i_1,\dotsc,i_t)$ is a sequence in $[n]$, we define its \textit{length} $l(I):=t$, and we set
	\[s_I:=s_{i_1}\dotsb s_{i_t}\quad \text{and}\quad \Delta_I:=\Delta_{i_1}\circ\dotsb\Delta_{i_t}\quad \text{and} \quad X_I:=X_{i_1}\dotsb X_{i_t}.\]
	We call $I$ a \textit{sequence of simple roots}.
	We say that $I$ is \textit{reduced} if $w(I):=s_{i_1}\dotsb s_{i_t}$ is reduced in $W$. For any reduced decomposition $w=s_{i_1}\dotsb s_{i_t}$ of $w\in W$, we call $I_w=(i_1,\dotsc,i_t)$ a \textit{reduced sequence} for $w$. 
	
	For each $w\in W$, fix a reduced sequence $I_w$. Set $w_0^{i,j}:=\underbrace{s_is_js_i\dotsb}_\text{$m_{i,j}$-times}$.

	\begin{prop}(\cite[Prop. 7.7]{CZZ1})\label{prop:basis-elements}
		The $R$-algebra $\mathbf{D}_F$ is free as a left $R\llbracket\Lambda\rrbracket_F$-submodule, with basis $\{X_{I_w}\}_{w\in W}$.
	\end{prop}
	
	The following remark is taken from \cite[Def. 4.2]{HMSZ}.
	\begin{rmk}
		For $\alpha\in\Sigma$, set $\kappa_\alpha:=\tfrac{1}{x_\alpha}+\tfrac{1}{x_{-\alpha}}\in\mcal{Q}^F$. In fact, we have that $\kappa_\alpha\in R\llbracket\Lambda\rrbracket_F$: \[0=x_{\alpha}+_F x_{-\alpha}=x_{\alpha}+x_{-\alpha}+\sum_{i,j\geq 1}a_{i,j}x_\alpha^ix_{-\alpha}^j\implies \kappa_\alpha=-\sum_{i,j\geq 1}a_{i,j}x_\alpha^{i-1}x_{-\alpha}^{j-1}\in R\llbracket\Lambda\rrbracket_F.\]
	\end{rmk}
	
	\begin{theo}(\cite[Theorem 7.9]{CZZ1})\label{theo:presentation}
		The elements $q\in R\llbracket\Lambda\rrbracket_F$ and the formal Demazure elements $X_i=X_{\alpha_i}$, where $\alpha_i\in\Delta$, satisfy the following relations:
		\begin{enumerate}
			\item\label{itm:theo1} $X_iq=\Delta_i(q)+s_i(q)X_i$;
			\item\label{itm:theo2} $X_i^2=\kappa_iX_i$, where $\kappa_i:=\tfrac{1}{x_{\alpha_i}}+\tfrac{1}{x_{-\alpha_i}}\in R\llbracket\Lambda\rrbracket_F$;
			\item\label{itm:theo3} $\underbrace{X_iX_jX_i\dotsb}_{\text{$m_{i,j}$-times}}-\underbrace{X_jX_iX_j\dotsb}_{\text{$m_{i,j}$-times}}=\sum\limits_{w<w_0^{i,j}}\eta_w^{i,j}X_{I_w}$, $\eta_w^{i,j}\in R\llbracket\Lambda\rrbracket_F$.
		\end{enumerate}
		Note that the ordering $<$ is with respect to the Bruhat order on $W$. These relations, together with the ring law in $R\llbracket\Lambda\rrbracket_F$ and the fact that the $X_i$ are $R$-linear form a complete set of relations in $\mathbf{D}_F$.
	\end{theo}
	
	\begin{rmk}\label{rmk:compute}
		The $\eta_w^{i,j}$ of Theorem \ref{theo:presentation} can be computed explicitly from the coefficients computed in \cite[Thm. 6.8]{HMSZ} using the formulas (\ref{itm:theo1}) and (\ref{itm:theo2}) of Theorem \ref{theo:presentation}.
	\end{rmk}
	
	\section{Formal Demazure operators}\label{section:Algebra-of-Demazure-operators}
	
	We assume $R\llbracket\Lambda\rrbracket_F$ is a formal group algebra satisfying Assumption \ref{assumption:sigma-regular} throughout this section, which is needed for \cref{theo:iso} and \cref{prop:cocommutative-coproduct}. In this section, we discuss the subalgebra $\mcal{D}_F$ of the endomorphism algebra of $R\llbracket\Lambda\rrbracket_F$ generated by formal Demazure operators and by multiplication by elements in $R\llbracket\Lambda\rrbracket_F$. The dual $\mcal{D}_F^*$ of $\mcal{D}_F$ will be used in \cref{section:Structure-of-oriented-cohomology-of-algebraic-groups} to study the oriented cohomology rings of the semisimple algebraic groups. We will closely follow \cite{CZZ1} in this section.

	\begin{defn}(\cite[\S 7]{CZZ1})\label{defn:endomorphisms}
		Let $\mcal{D}_F$ be the subalgebra of the algebra of $R$-linear endomorphisms of $R\llbracket\Lambda\rrbracket_F$ generated by the operators $\Delta_\alpha$, $\alpha\in\Sigma$, and by multiplication by elements in $R\llbracket\Lambda\rrbracket_F$.
	\end{defn}
	
	The proof of the following theorem uses  \cref{assumption:sigma-regular}.
	
	\begin{theo}(\cite[Theorem 7.10]{CZZ1})\label{theo:iso}
		The $R\llbracket\Lambda\rrbracket_F$-linear map $\phi:\mathbf{D}_F\to\mcal{D}_F$ sending $X_i\mapsto \Delta_i$ is an isomorphism of $R$-algebras.
	\end{theo}

	Fix a reduced sequence $I_w$ for each $w\in W$.
	
	\begin{theo}\label{theo:basis-of-algebra-of-operators}
		The algebra $\D_F$ is free as a left $R\llbracket\Lambda\rrbracket_F$-module with basis $\{\Delta_{I_w}\}_{w\in W}$.
	\end{theo}
	\begin{proof}
		This follows from Proposition \ref{prop:basis-elements} and Theorem \ref{theo:iso}.
	\end{proof}
	
	Let $q^*$ be the operator in $\D_F$ corresponding to multiplication by $q\in R\llbracket\Lambda\rrbracket_F$.
	
	\begin{prop}\label{prop:presentation-demazure-operators}
		The elements $q\in R\llbracket\Lambda\rrbracket_F$ and the formal Demazure operators $\Delta_i:=\Delta_{\alpha_i}$, where $\alpha_i\in\Delta$, satisfy the following relations:
		\begin{enumerate}
			\item $\Delta_i\circ q^*=\Delta_i(q)+(s_i(q))^*\circ \Delta_i$;
			\item $\Delta_i^2=\kappa_i^*\circ \Delta_i$, where $\kappa_i:=\tfrac{1}{x_{\alpha_i}}+\tfrac{1}{x_{-\alpha_i}}\in R\llbracket\Lambda\rrbracket_F$;
			\item $\underbrace{\Delta_i\circ \Delta_j\circ \Delta_i\dotsb}_{\text{$m_{i,j}$-times}}-\underbrace{\Delta_j\circ \Delta_i\circ \Delta_j\dotsb}_{\text{$m_{i,j}$-times}}=\sum\limits_{w<w_0^{i,j}}(\eta_w^{i,j})^*\circ \Delta_{I_w}$, $\eta_w^{i,j}\in R\llbracket\Lambda\rrbracket_F$.
		\end{enumerate}
		Note that the ordering $<$ is with respect to the Bruhat order on $W$. These relations, together with the ring law in $R\llbracket\Lambda\rrbracket_F$ and the fact that the $\Delta_i$ are $R$-linear form a complete set of relations in $\D_F$.
	\end{prop}
	\begin{proof}
		This follows from Theorem \ref{theo:presentation} and Theorem \ref{theo:iso}.
	\end{proof}

	\begin{defn}(\cite[Def. 4.7]{CZZ1})
		We define $R$-linear operators $B_i^{(j)}\colon R\llbracket\Lambda\rrbracket_F\to R\llbracket\Lambda\rrbracket_F$, where $j\in\{0,1,-1\}$ and $i\in[n]$, by 
		\[B_i^{(-1)}:=\Delta_i,\quad B_i^{(0)}:=s_i,\quad B_i^{(1)}:=\text{multiplication by $(-x_i)$:= $-x_{\alpha_i}$}.\]
	\end{defn}
	
	Let $I=(i_1,\dotsc,i_t)$ be a sequence of simple roots, and let $E$ be a subset of $[t]$. We denote by $I_E$ the subsequence of $I$ consisting of all $i_j$'s with $j\in E$.
	
	\begin{prop}\label{prop:cocommutative-coproduct}
		There is a cocommutative coalgebra structure on $\D_F$. The coproduct, denoted $\triangle$, satisfies the following: given any sequence of simple roots $I=(i_1,\dotsc,i_t)$, 
		\[\triangle(\Delta_I)=\sum_{E_1,E_2\subseteq [t]}p_{E_1,E_2}^I\Delta_{I|_{E_1}}\otimes \Delta_{I|_{E_2}},\]
		where $p_{E_1,E_2}^I:=B_1\circ\dotsb\circ B_t(1)$, and the operator $B_j:R\llbracket\Lambda\rrbracket_F\to R\llbracket\Lambda\rrbracket_F$ is defined by 
		\[
		B_j:=\begin{cases}
			B_{i_j}^{(1)}\circ B_{i_j}^{(0)}, &\mbox{if $j\in E_1\cap E_2$;}\\
			B_{i_j}^{(-1)}, &\mbox{if $j\notin E_1\cup E_2$;}\\
			B_{i_j}^{(0)}, &\mbox{otherwise}.
		\end{cases}
		\]	
	\end{prop}
	\begin{proof}
		This follows from \cite[Lemma 4.8, Proposition 9.5, and Theorem 10.4]{CZZ1}.
	\end{proof}
	
	Let $\D_F^*=\Hom_{R\llbracket\Lambda\rrbracket_F}(\D_F,R\llbracket\Lambda\rrbracket_F)$ be the $R\llbracket\Lambda\rrbracket_F$-linear dual of $\D_F$. The cocommutative $R\llbracket\Lambda\rrbracket_F$-coalgebra structure on $\D_F$ induces a commutative $R\llbracket\Lambda\rrbracket_F$-algebra structure on $\D_F^*$. Let $B^*:=\{\Delta_{I_w}^*\}_{w\in W}$ be the basis of $\D_F^*$ dual to $B:=\{\Delta_{I_w}\}_{w\in W}$. By Theorem \ref{theo:basis-of-algebra-of-operators} and Proposition \ref{prop:cocommutative-coproduct}, given any $\Delta_{I_v}\in B$, we can write 
	\begin{equation}\label{eq:coproduct-coefficients}\triangle(\Delta_{I_v})=\sum_{w,w'\in W}q_{I_w,I_{w'}}^{I_v}\Delta_{I_w}\otimes \Delta_{I_{w'}}\text{ },\quad q_{I_w,I_{w'}}^{I_v}\in R\llbracket \Lambda\rrbracket_F.\end{equation}
	Thus, products of elements in the dual basis satisfy
	\begin{equation}\label{eq:product-coefficients}
		\Delta_{I_w}^*\cdot \Delta_{I_{w'}}^*=\sum_{v\in W}q_{I_w,I_{w'}}^{I_v}\Delta_{I_v}^*.
	\end{equation}
	
	\begin{egg}\label{egg:rank-1}
		Suppose $\Sigma$ is a root system of rank $1$. Then $\Sigma=\{\alpha_1,-\alpha_1\}$, and $W=\{e,s_1\}$. We compute
		\[\triangle(\Delta_{1})=\mathbf{1}\otimes \Delta_{1}+ \Delta_{1}\otimes \mathbf{1}-x_{\alpha_1}\Delta_{1}\otimes \Delta_{1}\quad\text{and}\quad \triangle(\mathbf{1})=\mathbf{1}\otimes \mathbf{1}.\]
		Thus,
		\[\Delta_{1}^*\cdot \Delta_{1}^*=-x_{\alpha_1}\Delta_{1}^*\quad\text{and}\quad\mathbf{1}^*\cdot \Delta_{1}^*=\Delta_{1}^*\cdot \mathbf{1}^*=\Delta_{1}^*\quad \text{and}\quad \mathbf{1}^*\cdot \mathbf{1}^*=\mathbf{1}^*.\]
		
	\end{egg}
	
	\section{Algebraic groups and homogeneous line bundles}\label{section:line-bundles}
	
	In this section, we recall some facts about algebraic groups and state a key result (\cref{lem:isomorphic-line-bundles0}) concerning homogeneous line bundles on complete flag varieties.
	
	Fix an algebraically closed field $k$ of characteristic $0$. Suppose $G$ is a semisimple algebraic group over $k$, and let $T$ be a maximal torus sitting inside a Borel subgroup $B$ of $G$. The discussion from here up to \cref{lem:image-of-Borel} closely follows \cite{H75}.
	The set of characters $T\to\mathbb{G}_m$, where $\mathbb{G}_m=\mrm{Spec}(k[x,x^{-1}])$, is a lattice $\Lambda$ under pointwise multiplication. The set of nonzero weights for the adjoint action of $T$ on the Lie algebra $\mathfrak{g}$ of $G$ forms a root system $\Sigma$ inside the character lattice $\Lambda$. Let $\Lambda^\vee$ be the lattice dual to $\Lambda$, and let $\Sigma^\vee$ be the set of coroots in $\Lambda^\vee$ with respect to $\Sigma$. The quadruple $\Psi=(\Sigma,\Lambda,\Sigma^\vee,\Lambda^\vee)$ is a semisimple root datum. Moreover, $\Psi$ is independent of the choice of maximal torus $T$, and any semisimple root datum gives rise to a unique semisimple algebraic group $G$ over $k$ with that root datum. 
	
	The \textit{Weyl group} $W$ of $G$ with respect to $T$ is the quotient $N_G(T)/T$, where $N_G(T)$ is the normalizer of $T$ in $G$, and $W$ is naturally identified with the Weyl group of $\Psi$. The semisimple algebraic groups associated with irreducible semisimple root data are called \textit{simple}.
	If $G$ is simple and $\Lambda=\Lambda_r$ (resp. $\Lambda=\Lambda_w$), then the group $G$ is called \textit{adjoint} (resp. \textit{simply-connected}) and we will often denote this group by $G^{\text{ad}}$ (resp. $G^{\text{sc}}$). The following theorem gives a criterion for two simple algebraic groups to be isomorphic. 
	
	\begin{theo}(cf. \cite[Thm. 32.1]{H75})
		Let $G$ and $G'$ be simple algebraic groups over $k$ with isomorphic root systems and isomorphic fundamental groups. Then $G\simeq G'$.
	\end{theo}
	Below we list the simply-connected $G^{\text{sc}}$ and adjoint $G^{\text{ad}}$ groups for the classical Dynkin types:
	\begin{center}
		\begin{tabular}{ |c |c | c | }
			\hline
			Root system & $G^\text{sc}$ & $G^\text{ad}$ \\
			\hline\hline
			$A_n$ & $\mrm{SL}(n+1,k)$ & $\mrm{PGL}(n+1,k)$\\
			\hline
			$B_n$ & $\mrm{Spin}(2n+1,k)$ & $\mrm{SO}(2n+1,k)$\\
			\hline
			$C_n$ & $\mrm{Sp}(2n,k)$ & $\mrm{PSp}(2n,k)$\\
			\hline
			$D_n$ & $\mrm{Spin}(2n,k)$ & $\mrm{PGO}^+(2n,k)$ \\
			\hline
		\end{tabular}
	\end{center}

	The following two results will become useful later in this section.
	
	\begin{lem}(cf. \cite[\S 21.3, Cor. C]{H75})\label{lem:image-of-Borel}
		Let $\phi:G\to G'$ be a surjective morphism of connected algebraic groups over $k$. If $H$ is a Borel subgroup (resp. maximal torus; resp. maximal connected unipotent subgroup) in $G$, then $\phi(H)$ is a Borel subgroup (resp. maximal torus; resp. maximal connected unipotent subgroup) in $G'$.
	\end{lem}

	\begin{theo}(cf. \cite[Cor. 6.3.3 and Thm. 6.3.5]{S98} and \cite[Thm. 10.6]{B69})\label{theo:Borel-subgroup-factors}
		Let $G$ be a semisimple algebraic group over $k$, and let $B$ be a Borel subgroup of $G$. The set of unipotent elements $B_u$ in $B$ is a closed, connected, nilpotent, normal subgroup of $B$, and $B/B_u$ is a torus. Finally, if $T$ is any maximal torus of $G$ sitting in $B$, then $B=T\rtimes B_u$ (this is a semidirect product), and the restriction of the projection $B\to B/B_u$ to $T$ defines an isomorphism $T\simeq B/B_u$.
	\end{theo}

	Let $G$ be a semisimple algebraic group defined over $k$, and suppose $T$ is a maximal torus sitting inside a Borel subgroup $B$ of $G$. If $\lambda$ is a character of $T$, then $\lambda$ determines a one-dimensional irreducible representation $V_\lambda$ of $T$. By Theorem \ref{theo:Borel-subgroup-factors}, $B/B_u$ is a torus, and there is sequence of homomorphisms $B\to B/B_u\to T$, where the second homomorphism is an isomorphism. Thus, every character of $T$ lifts to a character $\lambda$ of $B$. The group $B$ acts on $V_\lambda$ by $b\cdot v=\lambda(b)^{-1}v$ for all $b\in B$ and $v\in V_\lambda$.
	
	\begin{theo}(cf. \cite[\S 8.5]{S98})\label{theo:line-bundle}
		Let $G$ be a semisimple algebraic group over $k$, and let $B$ be a Borel subgroup of $G$ containing a maximal torus $T$ of $G$.
		Let $\lambda$ be a character of $T$, and hence, of $B$. The set
		\[\mcal{L}(\lambda):=G\times_B V_\lambda=G\times V_\lambda/((g,v)\sim (gb,b^{-1}\cdot v))\]
		is an algebraic variety, and it is the total space of a line bundle over the complete flag variety $G/B$. The morphism $\pi\colon \mcal{L}(\lambda)\to G/B$ defining this line bundle sends $(g,v)B\mapsto gB$ for all $(g,v)B\in \mcal{L}(\lambda)$.
	\end{theo}

	\begin{rmk}
		There is a natural $G$-action on $\mcal{L}(\lambda)$ given by $h\cdot (g,v)=(hg, v)$ for all $h,g\in G$, $v\in V_\lambda$. The line bundle $\pi:\mcal{L}(\lambda)\to G/B$ is \textit{$G$-equivariant}:
		$\pi(h\cdot(g,v))=h\cdot \pi(g,v)$ for all $g,h\in G$, $v\in V_\lambda$. The line bundle $\mcal{L}(\lambda)$ is called a \textit{homogeneous line bundle}.
	\end{rmk}

	Let $G$ and $G_1$ be semisimple algebraic groups over $k$, and fix maximal tori $T$ and $T_1$ in $G$ and $G_1$, respectively, with corresponding root data $\Psi=(\Sigma,\Lambda,\Sigma^\vee,\Lambda^\vee)$ and $\Psi_1=(\Sigma_1,\Lambda_1,\Sigma_1^\vee,\Lambda_1^\vee)$.
	
	\begin{defn}(cf. \cite[\S 1]{S99}.)
		A \emph{central isogeny}\index{central isogeny of root data} of root data $\sigma':\Psi\to\Psi_1$ is an injective lattice homomorphism $\sigma:\Lambda\to \Lambda_1$, such that the dual map $\sigma^\vee\colon \Lambda_1^\vee\to\Lambda^\vee$ is also injective, and such that $\sigma$ induces a bijection $\sigma|_{\Sigma}\colon\Sigma\to \Sigma_1$ satisfying
		\[\sigma^\vee((\sigma(\alpha))^\vee)=\alpha^\vee,\quad \alpha\in\Sigma.\]
	\end{defn}
	
	\begin{defn}(cf. \cite[\S 1]{S99}.)
		A \emph{central isogeny} of algebraic groups $\phi:G_1\to G$ is a surjective morphism whose kernel is finite and central in $G_1$.
	\end{defn}
	
	\begin{prop}(cf. \cite[\S 1]{S99})\label{prop:isogeny-theorem}
		Let $\phi:G_1\to G$ be a central isogeny, mapping $T_1$ onto $T$. Then $\phi$ induces a central isogeny of root data $\sigma':\Psi\to \Psi_1$ such that $\sigma(\lambda)=\lambda\circ \phi|_{T_1}$ for all $\lambda\in \Lambda$. 
	\end{prop}
	
	Let $G^\text{sc}$ be the \textit{simply-connected cover} of $G$, i.e., $G^{\text{sc}}$ is the semisimple simply-connected algebraic group over $k$ with the same Dynkin type as $G$, such that there is a central isogeny $\phi:G^\text{sc}\to G$ (\cite[Exercise 10.1.4(1)]{S98}). Let $\Psi^{\text{sc}}=(\Sigma^{\text{sc}},\Lambda^{\text{sc}},(\Sigma^{\text{sc}})^\vee,(\Lambda^{\text{sc}})^\vee)$ be the root datum of $G^{\text{sc}}$. If $B^\text{sc}$ is a Borel subgroup of $G^\text{sc}$ with maximal unipotent connected subgroup $B_u^\text{sc}$, then, by Lemma \ref{lem:image-of-Borel}, $B:=\phi(B^\text{sc})$ is a Borel subgroup in $G$ with maximal unipotent connected subgroup $B_u:=\phi(B_u^\text{sc})$. By Theorem \ref{theo:Borel-subgroup-factors}, we can view $B^\text{sc}/B_u^\text{sc}$ and $B/B_u$ as maximal tori in $G^\text{sc}$ and $G$, respectively. Set $T^\text{sc}:=B^\text{sc}/B_u^\text{sc}$ and $T:=\phi(T^\text{sc})=B/B_u$.
	Let $\sigma:\Lambda\to \Lambda^\text{sc}$ be the injective homomorphism on character lattices induced by $\phi$. If $p:B\to T$ and $p^\text{sc}:B^\text{sc}\to T^\text{sc}$ are the canonical projections onto the quotients, then the following diagram commutes:
	\[
	\begin{tikzcd}
		B^\text{sc} \arrow[r,"p^\text{sc}"] \arrow[d,"\phi|_{B^\text{sc}}"'] & T^\text{sc} \arrow[d,"\phi|_{T^\text{sc}}"] \\
		B \arrow[r,"p"'] & T 
	\end{tikzcd} \ .
	\]
	Recall that we can lift a character of $T^\text{sc}$ (resp. $T$) to a character of $B^\text{sc}$ (resp. $B$) by composing the character on the right by $p^\text{sc}$ (resp. $p$). Given a character $\lambda$ of $T$, we have by Proposition \ref{prop:isogeny-theorem} that $\lambda\circ \phi|_{T^\text{sc}}=\sigma(\lambda)$. Thus, $\lambda\circ p\circ \phi|_{B^\text{sc}}=\lambda\circ \phi|_{T^\text{sc}}\circ p^\text{sc}=\sigma(\lambda)\circ p^\text{sc}$. From now on, we will abuse notation and denote the character $\lambda\circ p$ (resp. $\sigma(\lambda)\circ p^\text{sc}$) of $B$ (resp. $B^\text{sc}$) by $\lambda$ (resp. $\sigma(\lambda)$). Thus, $\lambda\circ \phi|_{B^\text{sc}}=\sigma(\lambda)$, which will be used in the proof of \cref{lem:isomorphic-line-bundles0}. 
	
	It follows from \cite[Theorem 22.6]{B69} that the central isogeny $\phi\colon G^{\text{sc}}\to G$ induces an isomorphism of flag varieties $\phi_{G^{\text{sc}}/B^{\text{sc}}}\colon G^{\text{sc}}/B^{\text{sc}}\xrightarrow{\simeq } G/B$. We will now prove the main result of this section.

	\begin{prop}\label{lem:isomorphic-line-bundles0}
		Suppose $\lambda^\text{sc}\in \Lambda^\text{sc}$, $\lambda\in \Lambda$, and $\lambda^\text{sc}=\sigma(\lambda)$. 
		\begin{enumerate}
			\item The pullback $(\phi_{G^{\text{sc}}/B^{\text{sc}}})^*(\mcal{L}(\lambda))\to G^{\text{sc}}/B^{\text{sc}}$ is isomorphic to the line bundle $\mcal{L}(\lambda^{\text{sc}})\to G^{\text{sc}}/B^{\text{sc}}$.
			\item The pushforward $(\phi_{G^{\text{sc}}/B^{\text{sc}}})_*(\mcal{L}(\lambda^{\text{sc}}))\to G/B$ is isomorphic to the line bundle $\mcal{L}(\lambda)\to G/B$.
		\end{enumerate} 
	\end{prop}
	\begin{proof}
		\begin{enumerate}
			\item Consider the pullback diagram:
			\[\begin{tikzcd}
				\phi_{G^{\text{sc}}/B^{\text{sc}}}^*(\mcal{L}(\lambda)) \arrow[r] \arrow[d] & \mcal{L}(\lambda) \arrow[d] \\
				G^{\text{sc}}/B^{\text{sc}} \arrow[r,"\phi_{G^{\text{sc}}/B^{\text{sc}}}"] & G/B
			\end{tikzcd},\]
			where $\phi_{G^{\text{sc}}/B^{\text{sc}}}^*(\mcal{L}(\lambda))$ denotes the pullback of the line bundle $\mcal{L}(\lambda)\to G/B$ along $\phi_{G^{\text{sc}}/B^{\text{sc}}}$.
			We compute
			\begin{align*}	\phi_{G^{\text{sc}}/B^{\text{sc}}}^*(\mcal{L}(\lambda))&=\{(g'B^{\text{sc}},(g,v)B)\in (G^{\text{sc}}/B^{\text{sc}})\times \mcal{L}(\lambda)\mid \phi_{G^{\text{sc}}/B^{\text{sc}}}(g'B^{\text{sc}})=gB\}\\&\simeq   (G^{\text{sc}}\times k)/B^{\text{sc}},
			\end{align*}
			where the action of $B^{\text{sc}}$ on $G^{\text{sc}}\times k$ is given by
			\[(b')^{-1}\cdot (g',v):=(g'b',\phi|_{B^{\text{sc}}}(b')^{-1}\cdot v),\quad b'\in B^{\text{sc}},\text{ } g'\in G^{\text{sc}}, \text{ } v\in k.\]
			By \cref{prop:isogeny-theorem}, we have \[\phi|_{B^{\text{sc}}}(b')^{-1}\cdot v=\lambda(\phi|_{B^{\text{sc}}}(b'))v=(\sigma(\lambda)(b'))v.\]
			It follows from \cite[Proposition 1.5]{I76} that
			\[\phi_{G^{\text{sc}}/B^{\text{sc}}}^*(\mcal{L}(\lambda))\simeq \mcal{L}(\sigma(\lambda)).\]
			\item As $\phi_{G^{\text{sc}}/B^{\text{sc}}}$ is an isomorphism, it follows that the pushforward $(\phi_{G^{\text{sc}}/B^{\text{sc}}})_*(\mcal{O}_{G^{\text{sc}}/B^{\text{sc}}})$ of the structure sheaf of $G^{\text{sc}}/B^{\text{sc}}$ is isomorphic to the structure sheaf $\mcal{O}_{G/B}$ of $G/B$. The projection formula now implies the result:
			\begin{align*}
				(\phi_{G^{\text{sc}}/B^{\text{sc}}})_*((\phi_{G^{\text{sc}}/B^{\text{sc}}})^*(\mcal{L}(\lambda)))&\simeq
				(\phi_{G^{\text{sc}}/B^{\text{sc}}})_*((\phi_{G^{\text{sc}}/B^{\text{sc}}})^*(\mcal{L}(\lambda))\otimes_{\mcal{O}_{G^{\text{sc}}/B^{\text{sc}}}} \mcal{O}_{G^{\text{sc}}/B^{\text{sc}}})\\&\simeq\mcal{L}(\lambda)\otimes_{\mcal{O}_{G/B}} ((\phi_{G^{\text{sc}}/B^{\text{sc}}})_*(\mcal{O}_{G^{\text{sc}}/B^{\text{sc}}}))
				\\&\simeq\mcal{L}(\lambda)\otimes_{\mcal{O}_{G/B}} \mcal{O}_{G/B}
				\\&\simeq\mcal{L}(\lambda).
			\end{align*}
		\end{enumerate}  
	\end{proof}

	\section{Oriented cohomology theories}\label{section:oriented-cohomology-theories}
	In this section, we will recall the notion of an oriented cohomology theory in the sense of Levine-Morel \cite{LM07}, explain the localization axiom of \cite{V19}, and review some facts from \cite{CPZ}. In this section, fix $k$, an algebraically closed field of characteristic $0$.
	
	Let $\textbf{Sch}_k$ be the category of separated schemes of finite type over $k$, and let $\textbf{Sm}_k$ be the full subcategory of $\textbf{Sch}_k$ consisting of schemes smooth and quasi-projective over $k$.
	An \textit{oriented cohomology theory} $\mrm{h}^*$ in the sense of \cite{LM07} is a contravariant functor from $\textbf{Sm}_k$ to the category of commutative, graded rings that satisfies certain `cohomological-type' axioms. Given an oriented cohomology theory $\mrm{h}^*$ and a vector bundle $E\to X$ of rank $n$ on a smooth variety $X$ over $k$, there is a set of Chern classes $c_i(E)\in\mrm{h}^i(X)$, $i\in\{0,\dotsc,n\}$. In particular, given two line bundles $\mcal{L}_1$ and $\mcal{L}_2$ on $X$, the first Chern classes satisfy Quillen's formula:
	\[c_1(\mcal{L}_1\otimes \mcal{L}_2)=c_1(\mcal{L}_1)+_F c_1(\mcal{L}_2),\]
	where $F$ is a one-dimensional commutative formal group law over $R=\mrm{h}^*(\mrm{Spec}(k))$. In particular, every oriented cohomology theory gives rise to a formal group law. As $k$ has characteristic $0$, there is a universal oriented cohomology theory called \textit{algebraic cobordism} $\Omega^*$ that satisfies the following universal property: for any oriented cohomology theory $\mrm{h}^*$, there is a unique morphism of oriented cohomology theories $\Omega^*\to\mrm{h}^*$ (\cite[Theorem 1.2.6]{LM07}). The formal group law corresponding to algebraic cobordism is the universal formal group law over the Lazard ring $R=\mathbb{L}$.
	
	\begin{egg}
		The Chow theory $\mrm{CH}^*$, which sends $X\in\textbf{Sm}_k$ to the Chow ring $\mrm{CH}^*(X)$, is an oriented cohomology theory, and its formal group law is the additive formal group law over $R=\mathbb{Z}$ (see \cref{ex:FGL}\ref{ex:additive:FGL}).
	\end{egg}
	
	\begin{egg}
		The Grothendieck $K^0$ functor, which sends $X\in\textbf{Sm}_k$ to the Grothendieck ring $K^0(X)$ of vector bundles on $X$, defines an oriented cohomology theory $K^*(-):=K^0(-)\otimes\mathbb{Z}[\beta,\beta^{-1}]$, where $\mrm{deg}(\beta)=-1$, whose formal group law is the multiplicative-periodic formal group law over $R=\mathbb{Z}[\beta,\beta^{-1}]$ (see \cref{ex:FGL}\ref{ex:mult:FGL}). 
	\end{egg}

	An oriented cohomology theory of the form $\mrm{h}^*=\Omega^*\otimes_\mathbb{L}R$, where the ring homomorphism $\mathbb{L}\to R$ defining the $\mathbb{L}$-module structure on the commutative ring $R$ is induced by a formal group law $F$ over $R$ (\cref{ex:FGL}\ref{itm:U-FGL}), is called a \textit{free oriented cohomology theory}. The Chow theory and $K^0(-)\otimes \mathbb{Z}[\beta,\beta^{-1}]$ are both examples of free oriented cohomology theories (\cite[Theorem 1.2.18 and Theorem 1.2.19]{LM07}).

	We will now explain the \textit{localization axiom}, which is a restrictive axiom that one can impose on an oriented cohomology theory. The Chow theory, the functor $K^0(-)\otimes \mathbb{Z}[\beta,\beta^{-1}]$, and algebraic cobordism all satisfy the localization axiom.
	
	\begin{defn}(\cite[Def. 2.1(EXCI)]{V19})\label{defn:localization}
		Suppose that $\mrm{h}^*$ is an oriented cohomology theory such that, for any $X\in\textbf{Sm}_k$ with closed subscheme $i\colon Z\to X$ and open complement $j\colon U\to X$, one has an exact sequence:
		\[\begin{tikzcd}
			\mrm{h}_*(Z)\arrow[r,"i_*"] &\mrm{h}_*(X)\arrow[r,"j^*"] &\mrm{h}_*(U)\arrow[r] & 0,
		\end{tikzcd}\]
		where, for any equidimensional $Y\in \textbf{Sm}_k$, we define $\mrm{h}_*(Y):=\mrm{h}^{\mrm{dim}_k(Y)-*}(Y)$, and for any quasi-projective scheme $Y\in \textbf{Sch}_k$ which is not assumed to be smooth, we define $\mrm{h}_*(Y):=\mrm{colim}_{V\to Y} \mrm{h}_*(V)$, where $V\to Y$ are projective morphisms from schemes $V\in \textbf{Sm}_k$ and where the transition maps in the colimit are push-forward maps. In this case, we say that $\mrm{h}^*$ satisfies the \textit{localization axiom}.
	\end{defn}

	Next, we briefly explain \cite[Assumption 13.2]{CPZ}.
	\begin{rmk}
		Let $G$ be a semisimple algebraic group over $k$, and suppose $T$ is a maximal torus sitting inside a Borel subgroup $B$ of $G$. Let $\Sigma$ be the root system of $T$, with Weyl group $W$, and fix a set of simple roots $\{\alpha_1,\dotsc,\alpha_n\}$ of $\Sigma$. Given a sequence of simple roots $I=(\alpha_{i_1},\dotsc,\alpha_{i_l})$, there is a smooth, projective variety $X_I$ called the \textit{Bott-Samelson} variety corresponding to $I$. Moreover, there is a morphism $q_I:X_I\to G/B$. If $\zeta_I:=(q_I)_*(1)\in\mrm{h}^*(G/B)$ is the push-forward of the fundamental class of the Bott-Samelson variety corresponding to $I$, then \cite[Assumption 13.2]{CPZ} says that, given a reduced sequence $I_w$ for each $w\in W$, the set $\{\zeta_{I_w}\}_{w\in W}$ is an $R$-basis of $\mrm{h}^*(G/B)$, where $R=\mrm{h}^*(\mrm{Spec}(k))$.
	\end{rmk}
	
	The following definition is taken from \cite[Def. 8.7]{CPZ}.
	\begin{defn}
		We say an oriented cohomology theory $\mrm{h}^*$ is \emph{weakly birationally invariant} if, for any proper birational morphism $f:Y\to X$ in $\textbf{Sm}_k$, the pushforward of the fundamental class $f_*(1_Y)$ is invertible.
	\end{defn}
	
	We will work under Assumption \ref{assumption:restrict-theory} for the rest of this paper, which is used for \cref{theo:Petrov} and \cref{theo:Gille}. In Example \ref{egg:examples}, we give examples of oriented cohomology theories $\mrm{h}^*$ that satisfy the criteria of Assumption \ref{assumption:restrict-theory}.
	\begin{assumption}\label{assumption:restrict-theory}
		$\mrm{h}^*$ is a weakly birationally invariant oriented cohomology theory satisfying \cite[Assumption 13.2]{CPZ} and the localization axiom of \cref{defn:localization}.
	\end{assumption}
	
	\begin{egg}{\textup{(see} \cite[Ex. 8.8 and Lem. 13.3]{CPZ}\textup{)}}\label{egg:examples}
		Examples of theories $\mrm{h}^*$ satisfying Assumption \ref{assumption:restrict-theory} include the Chow theory $\mrm{CH}^*$, the functor $K^*$, and algebraic cobordism $\mrm{\Omega}^*$.
	\end{egg}

	\section{Structure of oriented cohomology rings of semisimple algebraic groups}\label{section:Structure-of-oriented-cohomology-of-algebraic-groups} 
	
	In this section, we give a presentation for the oriented cohomology ring of a semisimple algebraic group in terms of formal Demazure operators, and we give an algorithm for computing these oriented cohomology rings. The main results are \cref{prop:cohomology-model}, \cref{lem:structure-of-cohomology}, and \cref{alg:only}.
	
	\begin{nota}
		We fix the following notation for the remainder of this paper.
		\begin{itemize}
			\item $k$ is an algebraically closed field of characteristic $0$.
			\item $\Psi=(\Sigma,\Lambda,\Sigma^\vee,\Lambda^\vee)$ is a semisimple root datum, with Weyl group $W$ and torsion index $\mathfrak{t}$.
			\item $\Delta=\{\alpha_i\}_{i=1}^n$ is a simple system for $\Sigma$, and $\{\lambda_i\}_{i=1}^n$ is a basis for $\Lambda$.
			\item $G$ is a semisimple algebraic group over $k$, with root datum $\Psi$.
			\item $B$ is a Borel subgroup of $G$, so that $G/B$ is a complete flag variety.
			\item $T$ is a maximal torus of $G$ sitting inside $B$, with character lattice $\Lambda$.
			\item $\phi:G^{\text{sc}}\to G$ is the simply-connected cover of $G$.
			\item $\Psi^{\text{sc}}=(\Sigma^{\text{sc}},\Lambda^{\text{sc}},(\Sigma^{\text{sc}})^\vee,(\Lambda^{\text{sc}})^\vee)$ is the root datum of $G^{\text{sc}}$, with Weyl group $W^\text{sc}$ and torsion index $\mathfrak{t}^\text{sc}$.
			\item $\sigma':\Psi\to \Psi^{\text{sc}}$ is the central isogeny of root data induced by $\phi$. 
			\item $\sigma:\Lambda\to \Lambda^\text{sc}$ is the injective group homomorphism on character lattices induced by $\sigma'$. By definition, $(\sigma(\alpha))^\vee(\sigma(\lambda))=\alpha^\vee(\lambda)$ for all $\alpha\in\Sigma$, $\lambda\in\Lambda$. 
			\item $B^\text{sc}$ is the Borel subgroup in $G^\text{sc}$ corresponding to $B$, and $T^\text{sc}$ is the maximal torus in $G^\text{sc}$ corresponding to $T$.
			\item $\mrm{h}^*$ is a weakly birationally invariant oriented cohomology theory satisfying Assumption \ref{assumption:restrict-theory}.
			\item $F(x,y)=x+y+\sum_{i,j\geq 1}a_{i,j}x^iy^j\in R\llbracket x,y\rrbracket$ is the one-dimensional commutative formal group law over $R=\mrm{h}^*(\mrm{Spec}(k))$ associated to $\mrm{h}^*$.
			\item The formal group algebras $R\llbracket\Lambda\rrbracket_F$ and $R\llbracket\Lambda^{\text{sc}}\rrbracket_F$ satisfy Assumption \ref{assumption:sigma-regular}.
			\item $\mcal{D}_F$ (resp. $\mcal{D}_F^\text{sc}$) is the algebra defined in Definition \ref{defn:endomorphisms} corresponding to $R\llbracket\Lambda\rrbracket_F$ (resp. $R\llbracket\Lambda^{\text{sc}}\rrbracket_F$).
			\item For each $w\in W$, $I_w=(\alpha_{i_1},\dotsc,\alpha_{i_k})$ is a reduced sequence of $w$.
			\item $\sigma(I_w):=(\sigma(\alpha_{i_1}),\dotsc,\sigma(\alpha_{i_k}))$ is the reduced sequence of the reflection $\sigma(w):=s_{\sigma(\alpha_{i_1})}\dotsb s_{\sigma(\alpha_{i_k})}\in W^\text{sc}$ corresponding to $I_w$.
		\end{itemize}
	\end{nota}
	
	The morphism $\sigma$ induces an injection of $R$-algebras $\sigma_*:R\llbracket\Lambda\rrbracket_F\to R\llbracket\Lambda^\text{sc}\rrbracket_F$ sending $x_\lambda\mapsto x_{\sigma(\lambda)}$. 
	
	\begin{lem}\label{lem:reflection-invariant}
		Let $s\in R\llbracket\Lambda\rrbracket_F$. Then 
		\[s_{\sigma(\alpha_i)}(\sigma_*(s))=\sigma_*(s_{\alpha_i}(s)).\]
	\end{lem}
	\begin{proof}
		If $s=x_{\lambda}$ for some $\lambda\in\Lambda$, then
		\begin{align*}
			s_{\sigma(\alpha_i)}(\sigma_*(x_\lambda))=x_{\sigma(\lambda)-(\sigma(\alpha_i))^\vee(\sigma(\lambda))\sigma(\alpha_i)}=x_{\sigma(\lambda)-\alpha_i^\vee(\lambda)\sigma(\alpha_i)}=\sigma_*(s_{\alpha_i}(x_\lambda)).
		\end{align*}
		By induction on the degree of monomials, for $u,v\in R\llbracket\Lambda\rrbracket_F$ of degrees greater than $0$, we have
		\[s_{\sigma(\alpha_i)}(\sigma_*(uv))=(s_{\sigma(\alpha_i)}(\sigma_*(u)))(s_{\sigma(\alpha_i)}(\sigma_*(v)))=\sigma_*(s_{\alpha_i}(uv)).\]
		By density of $R\llbracket\Lambda\rrbracket_F$, for all $s\in R\llbracket\Lambda\rrbracket_F$, we have $s_{\sigma(\alpha_i)}(\sigma_*(s))=\sigma_*(s_{\alpha_i}(s))$.
	\end{proof}
	
	Recall the augmentation map $\epsilon: R\llbracket\Lambda\rrbracket_F\to R$, which sends $x_\lambda\mapsto 0$ for all $\lambda\in\Lambda$ and fixes $R$. 
	\begin{lem}\label{lem:sigma-invariant}
		Let $s\in R\llbracket\Lambda\rrbracket_F$. Then 
		\[\epsilon\Delta_{I_w}(s)=\epsilon\Delta_{\sigma(I_w)}(\sigma_*(s)).\]
	\end{lem}
	\begin{proof}
		By Lemma \ref{lem:reflection-invariant}, we have
		\begin{align*}
			\Delta_{\sigma(\alpha_i)}(\sigma_*(s))&=\frac{\sigma_*(s)-s_{\sigma(\alpha_i)}(\sigma_*(s))}{x_{\sigma(\alpha_i)}}=\sigma_*\left(\frac{s-s_{\alpha_i}(s)}{x_{\alpha_i}}\right)=\sigma_*\left(\Delta_{\alpha_i}(s)\right).
		\end{align*}
		The result follows from the fact that $\sigma_*$ fixes $R$.
	\end{proof}

	The discussion from here up to Theorem \ref{theo:augmented-basis} closely follows \cite[\S 10]{CZZ1}.
	The augmentation map induces a map of $R$-modules $\epsilon_*:\mcal{D}_F\to \mrm{Hom}_R(R\llbracket\Lambda\rrbracket_F,R)$ given by $f\mapsto \epsilon\circ f$. Let $\epsilon\mcal{D}_F$ be the image of $\mcal{D}_F$ under the induced map. We can view $\epsilon\mcal{D}_F$ as an $R\llbracket\Lambda\rrbracket_F$-module, where the $R\llbracket\Lambda\rrbracket_F$-action is given by $s\cdot f:=\epsilon(s)f$ for all $s\in R\llbracket\Lambda\rrbracket_F$ and $f\in\epsilon\D_F$. There is an $R\llbracket\Lambda\rrbracket_F$-linear coproduct \[\triangle^\epsilon:\epsilon\D_F\to \epsilon\D_F\otimes_{R\llbracket\Lambda\rrbracket_F}\epsilon\D_F, \]
	such that 
	\[\triangle^\epsilon(f)(u\otimes v)=f(uv),\quad u,v\in R\llbracket\Lambda\rrbracket_F.\]
	The map $\epsilon_*:\D_F\to \epsilon\D_F$ is a morphism of $R\llbracket\Lambda\rrbracket_F$-coalgebras. Since Assumption \ref{assumption:sigma-regular} holds, Theorem \ref{theo:augmented-basis} below goes through:
	\begin{theo}(\cite[Thm. 5.4]{CPZ})\label{theo:augmented-basis}
		The augmentation $\epsilon\mcal{D}_F$ (resp. $\epsilon\mcal{D}_F^\text{\text{sc}}$) is free as a left $R$-module with basis $\{\epsilon\Delta_{I_w}\}_{w\in W}$ (resp. $\{\epsilon\Delta_{\sigma(I_w)}\}_{w\in W}$).
	\end{theo}
	
	Since $\epsilon_*$ is a morphism of $R\llbracket\Lambda\rrbracket_F$-coalgebras, we can compute the coproduct $\triangle^\epsilon$ in the basis $\{\epsilon\Delta_{I_w}\}_{w\in W}$ using Equation (\ref{eq:coproduct-coefficients}). 
	\begin{rmk}\label{rmk:iso-augmentation}
		By Lemma \ref{lem:reflection-invariant} and Lemma \ref{lem:sigma-invariant}, and in the notation of Equation (\ref{eq:coproduct-coefficients}), we have that 
		\[\epsilon(q_{I_w,I_{w'}}^{I_v})=\epsilon\left(q_{\sigma(I_w),\sigma(I_{w'})}^{\sigma(I_v)}\right).\] 
		Thus, by Theorem \ref{theo:augmented-basis}, $\sigma$ induces an isomorphism of $R$-coalgebras,
		\[\epsilon\D_F\to \epsilon\D_F^\text{sc},\quad \epsilon\Delta_{I_w}\mapsto  \epsilon\Delta_{\sigma(I_w)}.\] 
		Therefore, the map 
		\[(\epsilon\D_F)^*\to  (\epsilon\D_F^\text{sc})^*,\quad (\epsilon\Delta_{I_w})^*\mapsto   \left(\epsilon\Delta_{\sigma(I_w)}\right)^*\] 
		yields an isomorphism of $R$-algebras.
	\end{rmk}
	We recall three different maps of $R$-algebras:
	\begin{itemize}
		\item  $\xi:\D_F^*\to(\epsilon\D_F)^*$ (or $\xi^\text{sc}:(\D_F^\text{sc})^*\to(\epsilon\D_F^\text{sc})^*$) is the \textit{surjective} map of $R$-algebras, defined in \cite[Lemma 11.2]{CZZ1}, which sends $f$ to the map $\epsilon d\mapsto \epsilon f(d)$.
		\item  $c_{R\llbracket\Lambda\rrbracket_F}\colon R\llbracket\Lambda\rrbracket_F\to\D_F^*$ (resp. $c_{R\llbracket\Lambda^\text{sc}\rrbracket_F}\colon R\llbracket\Lambda^\text{sc}\rrbracket_F\to(\D_F^\text{sc})^*$) is the $R$-algebra map, defined in \cite[pp. 1215]{CZZ1}, which sends $s$ to $\mrm{ev}_s:=\sum_{w\in W}\Delta_{I_w}(s)\Delta_{I_w}^*$ (resp. $\mrm{ev}_s:=\sum_{w\in W}\Delta_{\sigma(I_w)}(s)\Delta_{\sigma(I_w)}^*$).
		\item   $c_R:R\llbracket\Lambda\rrbracket_F\to(\epsilon\D_F)^*$ (resp. $c_R^\text{sc}:R\llbracket\Lambda^\text{sc}\rrbracket_F\to(\epsilon\D_F^\text{sc})^*$) is the $R$-algebra map, defined in  \cite[Def. 6.1]{CPZ}, which sends $s$ to $\sum_{w\in W}\epsilon\Delta_{I_w}(s)(\epsilon\Delta_{I_w})^*$ (resp. $\sum_{w\in W}\epsilon\Delta_{\sigma(I_w)}(s)(\epsilon\Delta_{\sigma(I_w)})^*$).
	\end{itemize}
	
	These three maps are related by $c_R=\xi\circ c_{R\llbracket\Lambda\rrbracket_F}$ (resp. $c_R^{\text{sc}}=\xi^{\text{sc}}\circ c_{R\llbracket\Lambda^{\text{sc}}\rrbracket_F}$).
	
	\begin{rmk}\label{rmk:7.6}
		Recall the isomorphism of varieties $\phi_{G^{\text{sc}}/B^{\text{sc}}}\colon G^{\text{sc}}/B^{\text{sc}}\to G/B$ described in \cref{section:line-bundles}. This isomorphism of varieties induces an isomorphism of $R$-algebras $\phi_{G^{\text{sc}}/B^{\text{sc}}}^*\colon \mrm{h}^*(G/B)\to \mrm{h}^*(G^{\text{sc}}/B^{\text{sc}})$. Let $\lambda\in\Lambda$, and recall the line bundle $\mcal{L}(\lambda)$ on $G/B$. Let $c_1^{\mrm{h}^*}$ be the first Chern class with respect to $\mrm{h}^*$. By construction, the first Chern class $c_1^{\mrm{h}^*}$ commutes with pullbacks, so \cref{lem:isomorphic-line-bundles0} implies that
		\[\phi_{G^{\text{sc}}/B^{\text{sc}}}^*(c_1^{\mrm{h}^*}(\mcal{L}(\lambda)))=c_1^{\mrm{h}^*}(\phi_{G^{\text{sc}}/B^{\text{sc}}}^*(\mcal{L}(\lambda)))=c_1^{\mrm{h}^*}(\mcal{L}(\sigma(\lambda))).\]
		In other words, the elements $c_1^{\mrm{h}^*}(\mcal{L}(\lambda))\in \mrm{h}^*(G/B)$ and $c_1^{\mrm{h}^*}(\mcal{L}(\sigma(\lambda)))\in \mrm{h}^*(G^{\text{sc}}/B^{\text{sc}})$ are identified by the isomorphism $\phi_{G^{\text{sc}}/B^{\text{sc}}}^*$.
	\end{rmk}
	
	By \cite[Def. 10.2]{CPZ}, there is an $R$-algebra homomorphism called the \textit{characteristic map},
	\[\mathfrak{c}_{G^\text{sc}/B^\text{sc}}:R\llbracket\Lambda^\text{sc}\rrbracket_F\to \mrm{h}^*(G^\text{sc}/B^{\text{sc}}),\quad x_{\lambda^\text{sc}}\mapsto \mrm{c}_1^{\mrm{h}^*}(\mcal{L}(\lambda^\text{sc})),\quad \text{for all $\lambda^\text{sc}\in\Lambda^\text{sc}$},\]
	where $\mrm{c}_1^{\mrm{h}^*}(\mcal{L}(\lambda^\text{sc}))$ is the first Chern class of $\mcal{L}(\lambda^\text{sc})$ with respect to $\mrm{h}^*$.
	This gives an $R$-algebra homomorphism
	\[\mathfrak{c}_{G/B}:R\llbracket\Lambda\rrbracket_F\to R\llbracket\Lambda^\text{sc}\rrbracket_F\to \mrm{h}^*(G^\text{sc}/B^{\text{sc}}),\quad x_\lambda\mapsto \mrm{c}_1^{\mrm{h}^*}(\mcal{L}(\sigma(\lambda))),\]
	which we will also call a \textit{characteristic map}. By definition, $\mathfrak{c}_{G/B}=\mathfrak{c}_{G^\text{sc}/B^\text{sc}}\circ \sigma_*$.

	\begin{theo}\label{theo:Petrov}
		There is an $R$-algebra isomorphism $\theta:(\epsilon\D_F^\text{sc})^*\to \mrm{h}^*(G^\text{sc}/B^{\text{sc}})$ such that $\mathfrak{c}_{G/B}=\theta\circ c_R^\text{sc}\circ \sigma_*$. 
	\end{theo}
	\begin{proof}
		By \cite[Theorem 13.13]{CPZ}, there exists an $R$-algebra isomorphism $\theta:(\epsilon\D_F^\text{sc})^*\to \mrm{h}^*(G^\text{sc}/B^\text{sc})$ such that $\mathfrak{c}_{G^\text{sc}/B^\text{sc}}=\theta\circ c_R^\text{sc}$. The result follows by composing both sides on the right by $\sigma_*$.
	\end{proof}

	The proof of the following lemma uses \cite[Proposition 5.1]{GZ12}, which requires that $\mrm{h}^*$ satisfies the localization axiom.
	
	\begin{prop}\label{theo:Gille}
		There are $R$-algebra isomorphisms
		\begin{align*}
			\mrm{h}^*(G)&\simeq\mrm{h}^*(G/B)/(c_1^{\mrm{h}^*}(\mcal{L}(\lambda_1)),\dotsc,c_1^{\mrm{h}^*}(\mcal{L}(\lambda_n)))\\&\simeq  \mrm{h}^*(G^{\text{sc}}/B^{\text{sc}})/(c_1^{\mrm{h}^*}(\mcal{L}(\sigma(\lambda_1))),\dotsc,c_1^{\mrm{h}^*}(\mcal{L}(\sigma(\lambda_n)))).
		\end{align*}
	\end{prop}
	\begin{proof}
		This follows from \cite[Proposition 5.1]{GZ12} and \cref{lem:isomorphic-line-bundles0}. See also \cref{rmk:7.6}.
	\end{proof}

	Let $\mcal{I}_F$ (resp. $\mcal{I}_F^\text{sc}$) be the kernel of the augmentation map $\epsilon:R\llbracket\Lambda\rrbracket_F\to R$ (resp. $\epsilon:R\llbracket\Lambda^\text{sc}\rrbracket_F\to R$).
	Denote by  $\mcal{I}_F\D_F^*$ (resp. $\mcal{I}_F^\text{sc}(\D_F^\text{sc})^*$) the ideal in $\D_F^*$ (resp. $(\D_F^\text{sc})^*$)  generated by multiplication by elements in $\mcal{I}_F$ (resp. $\mcal{I}_F^\text{sc}$). Let $\mcal{C}_{R\llbracket\Lambda\rrbracket_F}$ be the ideal in $\D_F^*$ generated by the image of the restriction $c_{R\llbracket\Lambda\rrbracket_F}|_{\mcal{I}_F}$ of $c_{R\llbracket\Lambda\rrbracket_F}$ to $\mcal{I}_F$.
	
	\begin{theo}{\label{prop:cohomology-model}}
		There is an isomorphism of $R$-algebras 
		\[\mrm{h}^*(G)\simeq \D_F^*/(\mcal{C}_{R\llbracket\Lambda\rrbracket_F}+\mcal{I}_F\D_F^*).\]
	\end{theo} 
	\begin{proof}
		Using Theorem \ref{theo:basis-of-algebra-of-operators} and Theorem \ref{theo:augmented-basis}, we can verify that the kernel of the surjective map $\xi\colon\D_F^*\to(\epsilon \D_F)^*$ is $\mcal{I}_F\D_F^*$. Suppose 
		\[\xi\left(\sum_{w\in W}a_w\Delta_{I_w}^*\right)=0,\quad a_w\in R\llbracket\Lambda\rrbracket_F.\]
		Then, for all $v\in W$, we have
		\[0=\left(\xi\left(\sum_{w\in W}a_w\Delta_{I_w}^*\right)\right)\left(\epsilon\Delta_{I_v}\right)=\epsilon\left(\left(\sum_{w\in W}a_w\Delta_{I_w}^*\right)\left(\Delta_{I_v}\right)\right)=\epsilon(a_v).\]
		Since $\mrm{ker}(\epsilon)=\mcal{I}_F$, we have $\mrm{ker}(\xi)=\mathcal{I}_F\D_F^*$. Hence, the induced map $\bar{\xi}:\D_F^*/\mcal{I}_F\D_F^*\to (\epsilon \D_F)^*$ is an $R$-algebra isomorphism. 
		
		Define $\bar{c}_{R\llbracket\Lambda\rrbracket_F}:=\pi\circ c_{R\llbracket\Lambda\rrbracket_F}: R\llbracket\Lambda\rrbracket_F\to \D_F^*\to\D_F^*/\mcal{I}_F\D_F^*$, where $\pi:\D_F^*\to\D_F^*/\mcal{I}_F\D_F^*$ is the canonical projection. 
		Given $s\in R\llbracket\Lambda\rrbracket_F$, we see that
		\begin{align*}(c_R(s))(\epsilon\Delta_{I_v})&=\left(\sum_{w\in W}\epsilon\Delta_{I_w}(s)(\epsilon\Delta_{I_w})^*\right)(\epsilon\Delta_{I_v})\\&=\epsilon\Delta_{I_v}(s)=\bar{\xi}\left(\sum_{w\in W}\Delta_{I_w}(s)\Delta_{I_w}^*\right)(\epsilon\Delta_{I_v})=((\bar{\xi}\circ\bar{c}_{R\llbracket\Lambda\rrbracket_F})(s))(\epsilon\Delta_{I_v}).
		\end{align*}
		Therefore, $c_R=\bar{\xi}\circ\bar{c}_{R\llbracket\Lambda\rrbracket_F}$. By analogous reasoning, we have  $c_R^\text{sc}=\bar{\xi}^\text{sc}\circ\bar{c}_{R\llbracket\Lambda^\text{sc}\rrbracket_F}$

		The $R$-algebra isomorphism
		\[(\epsilon \D_F)^*\to (\epsilon \D_F^\text{sc})^*\]
		of Remark \ref{rmk:iso-augmentation} induces an $R$-algebra isomorphism
		\[\D_F^*/\mcal{I}_F\D_F^*\to(\D_F^\text{sc})^*/\mcal{I}_F^\text{sc}(\D_F^\text{sc})^* ,\quad \Delta_{I_w}^*+\mcal{I}_F\D_F^*\mapsto \Delta_{\sigma(I_w)}^*+\mcal{I}_F^\text{sc}(\D_F^\text{sc})^*.\]
		By Lemma \ref{lem:sigma-invariant}, under this isomorphism, $\mrm{im}(\bar{c}_{R\llbracket\Lambda^\text{sc}\rrbracket_F}\circ \sigma_*)$ in $(\mcal{D}_F^\text{sc})^*/\mcal{I}_F^\text{sc}(\mcal{D}_F^\text{sc})^*$ corresponds to $\mrm{im}(\bar{c}_{R\llbracket\Lambda\rrbracket_F})$ in $\mcal{D}_F^*/\mcal{I}_F\mcal{D}_F^*$.
		
		By \cref{theo:Petrov}, there is an $R$-algebra isomorphism $\theta:(\epsilon\D_F^\text{sc})^*\to \mrm{h}^*(G^\text{sc}/B^{\text{sc}})$ such that $\mathfrak{c}_{G/B}=\theta\circ c_R^\text{sc}\circ \sigma_*$. In particular, $(\D_F^\text{sc})^*/\mcal{I}_F^\text{sc}(\D_F^\text{sc})^*\simeq \mrm{h}^*(G^\text{sc}/B^{\text{sc}})$, and $\mathfrak{c}_{G/B}=\theta \circ\bar{\xi^\text{sc}}\circ \bar{c}_{R\llbracket\Lambda^\text{sc}\rrbracket_F}\circ \sigma_*$.
		
		By \cref{theo:Gille}, there is an $R$-algebra isomorphism
		\[\mrm{h}^*(G)\simeq \mrm{h}^*(G^\text{sc}/B^{\text{sc}})/(c_1^{\mrm{h}^*}(\mcal{L}(\sigma(\lambda_1))),\dotsc,c_1^{\mrm{h}^*}(\mcal{L}(\sigma(\lambda_n))))=\mrm{h}^*(G^\text{sc}/B^{\text{sc}})/(\mrm{im}(\mathfrak{c}_{G/B}|_{\mcal{I}_F})),\]
		where $\mathfrak{c}_{G/B}|_{\mcal{I}_F}$ is the restriction of $\mathfrak{c}_{G/B}$ to $\mcal{I}_F$.
		Therefore, through the identification of $\mathfrak{c}_{G/B}$ and $\bar{c}_{R\llbracket\Lambda^\text{sc}\rrbracket_F}\circ \sigma_*$ via the isomorphism $\theta\circ \bar{\xi^\text{sc}}$, we get 
		\begin{align*}\mrm{h}^*(G)\simeq \mrm{h}^*(G^\text{sc}/B^{\text{sc}})/(\mrm{im}(\mathfrak{c}_{G/B}|_{\mcal{I}_F}))&\simeq \left((\mcal{D}_F^\text{sc})^*/\mcal{I}_F^\text{sc}(\mcal{D}_F^\text{sc})^*\right)/(\mrm{im}((\bar{c}_{R\llbracket\Lambda^\text{sc}\rrbracket_F}\circ \sigma_*)|_{\mcal{I}_F}))\\&\simeq \left(\mcal{D}_F^*/\mcal{I}_F\mcal{D}_F^*\right)/(\mrm{im}(\bar{c}_{R\llbracket\Lambda\rrbracket_F}|_{\mcal{I}_F}))\\&\simeq \D_F^*/(\mcal{I}_F\D_F^*+\mcal{C}_{R\llbracket\Lambda\rrbracket_F}).\end{align*}\qedhere
	\end{proof}
	
	Let $R\langle B^*\rangle $ be the free $R$-algebra generated by the symbols $\Delta_{I_w}^*$, $w\in W$. Let $\mcal{M}$ (for `multiplication') be the ideal in $R\langle B^*\rangle$ generated by the elements $\Delta_{I_w}^*\cdot \Delta_{I_{w'}}^*-\sum_{v\in W}\epsilon (q_{I_w,I_{w'}}^{I_v})\Delta_{I_v}^*$ over all $w,w'\in W$, where the $q_{I_w,I_{w'}}^{I_v}\in R\llbracket \Lambda\rrbracket_F$ are defined after \cref{eq:coproduct-coefficients}. Thus, $R\langle B^*\rangle/\mcal{M}\simeq\mcal{D}_F^*/\mcal{I}_F\mcal{D}_F^*$. Recall the $\mathbb{Z}$-basis $\{\lambda_i\}_{i=1}^n$ of $\Lambda$, and let $\mcal{A}'$ (for `addition') be the ideal in $R\langle B^*\rangle/\mcal{M}$ generated by the elements $\sum_{w\in W}\epsilon\left(\Delta_{I_w}(x_{\lambda_i})\right)\Delta_{I_w}^*$ over all $i\in[n]$. Let $\mcal{A}$ be the ideal in $R\langle B^*\rangle$ generated by the elements $\sum_{w\in W}\epsilon\left(\Delta_{I_w}(x_{\lambda_i})\right)\Delta_{I_w}^*$ over all $i\in[n]$. The following corollary gives a presentation for $\mrm{h}^*(G)$ in terms of generators and relations.
	\begin{cor}{\label{lem:structure-of-cohomology}}
		There is an isomorphism of $R$-algebras 
		\[\mrm{h}^*(G)\simeq R\langle B^*\rangle/(\mcal{M}+\mcal{A})\text{.}\]
	\end{cor}
	\begin{proof}
		By \cref{prop:cohomology-model}, we have
		\[\mrm{h}^*(G)\simeq \D_F^*/(\mcal{C}_{R\llbracket\Lambda\rrbracket_F}+\mcal{I}_F\D_F^*).\] As  $c_{R\llbracket\Lambda\rrbracket_F}$ is a ring homomorphism and $\Lambda$ is spanned by the $\lambda_i$, $i\in [n]$, we see that $\mcal{C}_{R\llbracket\Lambda\rrbracket_F}$ is generated by the elements $\mrm{ev}_{x_{\lambda_i}}$, $i\in[n]$. Thus, the ideal $(\mcal{C}_{R\llbracket\Lambda\rrbracket_F}+\mcal{I}_F\D_F^*)/\mcal{I}_F\D_F^*$ in $\D_F^*/\mcal{I}_F\D_F^*$ corresponds to $\mcal{A}'$ under the identification $\D_F^*/\mcal{I}_F\D_F^*\simeq R\langle B^*\rangle/\mcal{M}$. Therefore,
		\[
		\D_F^*/(\mcal{C}_{R\llbracket\Lambda\rrbracket_F}+\mcal{I}_F\D_F^*)\simeq (\D_F^*/\mcal{I}_F\D_F^*)/\left((\mcal{C}_{R\llbracket\Lambda\rrbracket_F}+\mcal{I}_F\D_F^*)/\mcal{I}_F\D_F^*\right)\simeq (R\langle B^*\rangle/\mcal{M})/\mcal{A}'\simeq R\langle B^*\rangle/(\mcal{M}+\mcal{A}).	
		\]
	\end{proof}

	We now describe a technique to present $\mrm{h}^*(G)$ in terms of the generators and relations of \cref{lem:structure-of-cohomology}.
	
	\begin{alg}\label{alg:only}
		\begin{enumerate}
			\item Compute the coproduct coefficients $p_{E_1,E_2}^{I_v}$ of \cref{prop:cocommutative-coproduct}:
			\[\triangle(\Delta_{I_v})=\sum_{E_1,E_2\subseteq [l(v)]}p_{E_1,E_2}^{I_v}\Delta_{I_v|_{E_1}}\otimes \Delta_{I_v|_{E_2}}.\] 
			\item Express each $\Delta_{I_v|_{E_1}}$, $\Delta_{I_v|_{E_2}}$ in the formula above as an $R\llbracket\Lambda\rrbracket_F$-linear combination of the $\Delta_{I_w}$, using \cref{prop:presentation-demazure-operators}. Thus, 
			\[\triangle(\Delta_{I_v})=\sum_{w,w'\in W}q_{I_w,I_{w'}}^{I_v}\Delta_{I_{w}}\otimes \Delta_{I_{w'}},\quad q_{I_w,I_{w'}}^{I_v}\in R\llbracket\Lambda\rrbracket_F.\]
			\item Express each coefficient $q_{I_w,I_{w'}}^{I_v}$ as the (unique) power series over $R$ in the variables $x_{\lambda_i}$. The constant term in this power series is precisely $\epsilon(q_{I_w,I_{w'}}^{I_v})$.
			\item Express each $\Delta_{I_w}(x_{\lambda_j})$ as the (unique) power series over $R$ in the variables $x_{\lambda_i}$. The constant term in this power series is precisely $\epsilon(\Delta_{I_w}(x_{\lambda_i}))$.
			\item Now plug $\epsilon(q_{I_w,I_{w'}}^{I_v})$ and $\epsilon(\Delta_{I_w}(x_{\lambda_i}))$ into the relations $\mathcal{M}$ and $\mathcal{A}$ of \cref{lem:structure-of-cohomology}.
		\end{enumerate}
	\end{alg}
	
	\section{Computations for the groups of types $A_1$, $A_2$, and $B_2$}\label{section:Examples-in-rank-2}
	
	In this section, we compute minimal presentations for $\mrm{h}^*(G)$ in terms of generators and relations, where $G$ is an adjoint or simply-connected algebraic group of type $A_1$, $A_2$, or $B_2$. The main results of this section are summarized in \cref{tab:title}. Note that the column $K^0(G)$ (i.e., the Grothendieck ring of vector bundles on $G$) is the specialization of the general $\mrm{h}^*(G)$ column at $a_{11}=1$ and $a_{ij}=0$ for all $i+j\geq 3$; and the column $\mrm{CH}^*(G)$ (i.e., the Chow ring of $G$) is the specialization of the $\mrm{h}^*(G)$ column at $a_{ij}=0$ for all $i+j\geq 2$. 
	
	\begin{defn}
		If $G$ is simple, simply-connected (resp. adjoint), and of Dynkin type $\mcal{D}$, let $\Lambda^\text{sc}_{\mcal{D}}$ (resp. $\Lambda^\text{ad}_{\mcal{D}}$) be the character lattice for $T$.
	\end{defn}
	\begin{table}[!htb]
		\caption {Oriented cohomology rings of the adjoint/simply-connected groups for $A_1$, $A_2$, $B_2$} \label{tab:title} 
		\begin{center}
			\begin{tabular}{ |c |c |c |c |c |c |c |c | }
				\hline
				Rank & $\Sigma$ & $\Lambda$ & $G$ & $\mrm{h}^*(G)$ & $\Omega^*(G)$ & $K^0(G)$ & $\mathrm{CH}^*(G)$ \\
				\hline
				\hline
				$1$ &$A_1$ &  $\Lambda_{A_1}^{\text{ad}}$ & $\mrm{PGL}(2,k)$ & $\tfrac{R[x]}{(2x,x^2)}$ & $\tfrac{\mathbb{L}[x]}{(2x,x^2)}$ & $\tfrac{\mathbb{Z}[x]}{(2x,x^2)}$ & $\tfrac{\mathbb{Z}[x]}{(2x,x^2)}$ \\
				&& $\Lambda_{A_1}^{\text{sc}}$ & $\mrm{SL}(2,k)$ & $R$ & $\mathbb{L}$ & $\mathbb{Z}$ & $\mathbb{Z}$ \\
				\hline
				$2$ &$A_2$ &$ \Lambda_{A_2}^{\text{ad}}$ & $\mrm{PGL}(3,k)$ & $\tfrac{R[x]}{(3x,x^3)}$ & $\tfrac{\mathbb{L}[x]}{(3x,x^3)}$ & $\tfrac{\mathbb{Z}[x]}{(3x,x^3)}$ & $\tfrac{\mathbb{Z}[x]}{(3x,x^3)}$ \\
				&& $\Lambda_{A_2}^{\text{sc}}$ & $\mrm{SL}(3,k)$ & $R$ & $\mathbb{L}$ & $\mathbb{Z}$ & $\mathbb{Z}$\\
				$2$ &$B_2$ & $\Lambda_{B_2}^{\text{ad}}$ & $\mrm{SO}(5,k)$ & $\tfrac{R[x]}{(2x-a_{11}x^2,2x^2,x^4)}$ & $\tfrac{\mathbb{L}[x]}{(2x-a_{11}x^2,2x^2,x^4)}$ & $\tfrac{\mathbb{Z}[x]}{(2x-x^2,x^3)}$ & $\tfrac{\mathbb{Z}[x]}{(2x,x^4)}$ \\
				&& $\Lambda_{B_2}^\text{sc}$ & $\mrm{Spin}(5,k)$ & $R$ & $\mathbb{L}$ & $\mathbb{Z}$ & $\mathbb{Z}$\\
				\hline
			\end{tabular}
		\end{center} 
	\end{table}
	
	\begin{rmk}
		The group $\mrm{Spin}(5,k)$ is isomorphic to the group $\mrm{Sp}(4,k)$, since they have isomorphic root systems ($C_2\simeq B_2$) and fundamental groups. Similarly, $\mathrm{SO}(5,k)$ is isomorphic to $\mrm{PSp}(4,k)$.
	\end{rmk}

	If $G$ has rank $1$ or $2$, then the root datum corresponding to $G$ has rank $1$ or $2$, respectively. The Python code in \cite{G} generates the relations $\mcal{M}$ and $\mcal{A}$ of \cref{lem:structure-of-cohomology} for the groups in \cref{tab:title}. To obtain the presentations that appear in \cref{tab:title}, we simplify the list of relations generated by the code by hand. We write out the full details of this computation for the groups $\mrm{PGL}(2,k)$, $\mrm{PGL}(3,k)$, $\mrm{SL}(2,k)$, and $\mrm{SL}(3,k)$ in \cref{theo:A1xA1} and \cref{theo:A2}. In \cref{theo:B2} (i.e., in computing the ring presentations for  $\mrm{h}^*(\mrm{SO}(5,k))$ and $\mrm{h}^*(\mrm{Spin}(5,k))$), we omit some details, as the computations are lengthier while the method is the same as for \cref{theo:A1xA1} and \cref{theo:A2}.

	We begin by computing minimal presentations in terms of generators and relations for the rank $1$ groups.
	
	\begin{egg}\label{theo:A1xA1}
		We will show that
		\[\mrm{h}^*(\mrm{PGL}(2,k))\simeq R[x]/(2x,x^2);\quad \quad \mrm{h}^*(\mrm{SL}(2,k)) \simeq R.\]
		The root system for both groups considered in this theorem is $A_1$. Choose a simple root $\alpha$. From the code in \cite{G}, the ideal $\mcal{M}$ is generated by the following relation:
		\[
		\begin{tabular}{ l l  }
			$\text{(1) }  \Xa\Xa= 0$.
		\end{tabular}
		\]
		The ideal $\mcal{A}$ depends on the character lattice $\Lambda$. There are two cases to consider, depending on $\Lambda$.\newline
		
		\underline{Case 1: $\Lambda=\Lambda_{A_1}^{\text{sc}}$.}  Let $\lambda_1=\tfrac{1}{2}\alpha$, as in \cref{egg:A1xA1}. The ideal $\mcal{A}$ is generated by the relation:
		\[
		\begin{tabular}{ l l  }
			$\text{(a) } 0=\mrm{ev}_{x_{\lambda_1}}=\Xa$.
		\end{tabular}
		\]
		It is immediate that
		\[\mrm{h}^*(\mrm{SL}(2,k))\simeq R.\]

		\underline{Case 2: $\Lambda=\Lambda_{A_1}^{\text{ad}}$.} The ideal $\mcal{A}$ is generated by the following relation:
		
		\[
		\begin{tabular}{ l l  }
			$\text{(a) } 0=\mrm{ev}_{x_{\alpha}}=2\Xa$.
		\end{tabular}
		\]
		
		Therefore,
		\[\mrm{h}^*(\mrm{PGL}(2,k))\simeq R[x]/(2x,x^2),\quad \text{via} \quad \Delta_\alpha^*\mapsto x.\]
	\end{egg}

	If $G$ has rank $2$, then fix a basis $\{\lambda_1,\lambda_2\}$ for the lattice $\Lambda$, and a simple system $\Delta=\{\alpha,\beta\}$ for the root system $\Sigma$. If $\Lambda$ is simply-connected, we assume that the $\lambda_i$ are the fundamental weights with respect to $\Delta$. 
	The Weyl group $W$ is generated by the simple reflections $s_\alpha$ and $s_\beta$, and  $s_\alpha s_\beta$ has order $m\in \{2,3,4,6\}$ in $W$.
	If $w\in W$ is not the longest word, then there is exactly one choice for the sequence $I_w$. If $w\in W$ is the longest word, then there are exactly two choices for the sequence $I_w$. If $w$ is the longest word in $W$, we choose $I_w=\underbrace{(\alpha, \beta, \alpha,\dotsb)}_{\text{$m$ times}}$. Thus, the set $Y=\{I_w\}_{w\in W}$ is 
	\[Y=\{1,\alpha,\beta,(\alpha,\beta),(\beta,\alpha),\underbrace{(\alpha,\beta,\alpha\dotsb)}_{\text{$m-1$ times}},\underbrace{(\beta, \alpha, \beta,\dotsb)}_{\text{$m-1$ times}},\underbrace{(\alpha, \beta, \alpha,\dotsb)}_{\text{$m$ times}}\}.\] 
	
	\begin{egg}\label{theo:A2}
		We will show that
		\[\mrm{h}^*(\mrm{PGL}(3,k))\simeq R[x]/(3x,x^3);\quad \quad \mrm{h}^*(\mrm{SL}(3,k))\simeq R.\]
		The root system for both of the groups considered in this theorem is $A_2$. From the code in \cite{G}, the ideal $\mcal{M}$ is generated by the following relations:
		\[
		\begin{tabular}{ l l  }
			$\text{(1) }  \Xa\Xa= \Xba$;\textcolor{white}{} & $\text{(2) }  \Xa\Xb=\Xab+\Xba-a_{11}\Xaba$; \\
			$\text{(3) }  \Xa\Xab=\Xaba$; & $\text{(4) } \Xa\Xba=0$; \\
			$\text{(5) }  \Xa\Xaba=0$; & $\text{(6) } \Xb\Xb=\Xab$;  		\\
			$\text{(7) }  \Xb\Xab=0$; & $\text{(8) }  \Xb\Xba=\Xaba$;   \\
			$\text{(9)  }   \Xb\Xaba=0$; & $\text{(10) }  \Xab\Xab=0$;  		 \\
			$\text{(11) }  \Xab\Xba=0$; & $\text{(12) }  \Xab\Xaba=0$;  	\\
			$\text{(13) }  \Xba\Xba=0$; & $\text{(14) }  \Xba\Xaba=0$; \\
			$\text{(15) }  \Xaba\Xaba=0$. & 
		\end{tabular}
		\]
		The ideal $\mcal{A}$ depends on the character lattice $\Lambda$. There are two cases to consider, depending on $\Lambda$.\newline
		
		\underline{Case 1: $\Lambda=\Lambda_{A_2}^{\text{sc}}$.} As in \cref{egg:A2}, choose $\lambda_1$ and $\lambda_2$ such that 
		\[\alpha^\vee(\lambda_1)=1\quad \text{and}\quad\alpha^\vee(\lambda_2)=0\quad \text{and}\quad\beta^\vee(\lambda_1)=0\quad \text{and}\quad\beta^\vee(\lambda_2)=1.\]
		Then $s_\alpha(\lambda_1)=\lambda_1-\alpha$ and $s_\alpha(\lambda_2)=\lambda_2$ and $s_\beta(\lambda_1)=\lambda_1$ and $s_\beta(\lambda_2)=\lambda_2-\beta$. Thus,
		\[\Delta_\alpha(x_{\lambda_2})=\Delta_\beta(x_{\lambda_1})=0.\]
		In addition,
		\[\epsilon\Delta_\alpha(x_{\lambda_1})=\epsilon\left(\tfrac{x_{\lambda_1}-x_{\lambda_1-\alpha}}{x_\alpha}\right)=\epsilon\left(\tfrac{x_{\lambda_1}-(x_{\lambda_1}+_Fx_{-\alpha})}{x_\alpha}\right)=1.\]
		Similarly, $\epsilon\Delta_\beta(x_{\lambda_2})=1$. 
		Therefore, there exist $r_{I_1}^{I_2}\in R$ such that the ideal $\mcal{A}$ is generated by the relations
		\begin{enumerate}[label=(\alph*)]
			\item $0=\mrm{ev}_{x_{\lambda_1}}=\Xa+r_1^{\ab}\Xab+r_1^{\ba}\Xba+r_1^{\aba}\Xaba$,
			\item $0=\mrm{ev}_{x_{\lambda_2}}=\Xb+r_2^{\ab}\Xab+r_2^{\ba}\Xba+r_2^{\aba}\Xaba$.
		\end{enumerate}
		
		Relations (a), (3), (10), (11), and (12) imply: (i) $\Delta_{\aba}^*=0$. 
		
		Relations (i), (a), (1), (3), and (4) imply: (ii) $\Delta_{\ba}^*=0$.
		
		Relations (i), (ii), (b), (6), and (7) imply: (iii) $\Delta_{\ab}^*=0$.
		
		Relations (i), (ii), (iii), (a), and (b) imply: (iv) $\Delta_\alpha^*=\Delta_\beta^*=0$.
		
		\cref{prop:cohomology-model} and \cref{lem:structure-of-cohomology} tell us that the relations (a)-(b) and (1)-(15) form a complete set of relations in $\mrm{h}^*(\mrm{SL}(3,k))$. Therefore,
		\[\mrm{h}^*(\mrm{SL}(3,k))\simeq R.\]
		\newline
		\underline{Case 2: $\Lambda=\Lambda_{A_2}^{\text{ad}}$.} The output of the code in \cite{G} tells us that $\mcal{A}$ is generated by the following relations:
		\begin{enumerate}[label=(\alph*)]
			\item $0=\mrm{ev}_{x_\alpha}=2\Xa-\Xb-2a_{11}\Xab+a_{11}\Xba-a_{11}^2\Xaba$,
			\item $0=\mrm{ev}_{x_\beta}=-\Xa+2\Xb+a_{11}\Xab-2a_{11}\Xba-(a_{11}^2+3a_{12})\Xaba$.
		\end{enumerate}
		\noindent Relations (b), (3), (7), (10), (11), and (12) imply: (i) $\Xaba=0$.
		
		\noindent Relations (i), (b), (2), (6), (7), and (8) imply: \hspace{0.43cm} (ii) $\Xab=\Xba$.
		
		\noindent Relations (i), (ii), (a), and (b) imply: \hspace{1.7cm} (iii) $3\Xb=3a_{11}\Xab$.
		
		\noindent Relations (iii), (6), and (7) imply:\hspace{2.4cm} (iv) $3\Xb=3\Xab=0$.
		
		\noindent Relations (ii), (b), (6), and (7) imply: \[\text{(v) } \Xb\Xb\Xb=0;\quad  \text{(vi) } \Xba=\Xab=\Xb\Xb; \quad \text{(vii)  } \Xa=2\Xb-a_{11}\Xb\Xb.\]

		In summary, $\mrm{h}^*(\mrm{PGL}(3,k))$ is generated as a ring by $\Xb$, and $\Xb$ satisfies $3\Xb=\Xb\Xb\Xb=0$.
		Moreover, the relations (a)-(b) and (1)-(15) are consequences of the relations $3\Xb=\Xb\Xb\Xb=0$. By \cref{prop:cohomology-model} and \cref{lem:structure-of-cohomology}, the relations (a)-(b) and (1)-(15) form a complete set of relations for $\mrm{h}^*(\mrm{PGL}(3,k))$. Therefore,  
		\[\mrm{h}^*(\mrm{PGL}(3,k))\simeq  R[x]/(3x,x^3),\quad\text{via}\quad \Xb\mapsto x.\qedhere\]
	\end{egg}
	
	\begin{egg}\label{theo:B2}
		We will show that 
		\[\mrm{h}^*(\mrm{SO}(5,k))\simeq R[x]/(2x-a_{11}x^2,2x^2,x^4);\quad \mrm{h}^*(\mrm{Spin}(5,k))\simeq R.\]
		The root system for both of the groups considered in this theorem is $B_2$.  The output of the code in \cite{G} tells us that the ideal $\mcal{M}$ is generated by the relations below. Note: we write $q_{I_w,I_{w'}}^{I_v}$ below, but we really mean $\epsilon(q_{I_w,I_{w'}}^{I_v})\in R$. We make this notational change to reduce the clutter in the displayed relations. We do not need to know the explicit formulas for the coefficients denoted $q_{I_w,I_{w'}}^{I_v}$ in this computation, as the computation will show that the variables they sit in front of equal $0$.
		\begin{enumerate}
			\item $\Xa\Xa= \Xba+q_{\aaa,\aaa}^{\abab}\Xabab$,
			\item $\Xa\Xb=\Xab+\Xba-a_{11}\Xaba-a_{11}\Xbab+q_{\aaa,\bbb}^{\abab}\Xabab$,
			\item $\Xa\Xab=\Xaba+\Xbab+q_{\aaa,\ab}^{\abab}\Xabab$,
			\item $\Xa\Xba=\Xaba+q_{\aaa,\ba}^{\abab}\Xabab$,
			\item $\Xa\Xbab=\Xabab$,
			\item $\Xb\Xb=2\Xab-a_{11}\Xbab+q_{\bbb,\bbb}^{\abab}\Xabab$,
			\item $\Xb\Xab=\Xbab+q_{\bbb,\ab}^{\abab}\Xabab$,
			\item $\Xb\Xba=2\Xaba+\Xbab+q_{\bbb,\ba}^{\abab}\Xabab$,
			\item $\Xb\Xaba=\Xabab$,
			\item $\Xab\Xab=q_{\ab,\ab}^{\abab}\Xabab$,
			\item $\Xab\Xba=q_{\ab,\ba}^{\abab}\Xabab$,
			\item $\Xba\Xba=q_{\ba,\ba}^{\abab}\Xabab$,
			\item $\Delta^*_{I_w}\Delta^*_{I_{w'}}=0$ for all two-element subsets $\{I_w,I_{w'}\}$ of $\{I_w\}_{w\in W}$ that do not appear above.
		\end{enumerate}	
		The ideal $\mcal{A}$ depends on the character lattice $\Lambda$. There are two cases to consider, depending on $\Lambda$.\newline

		\underline{Case 1: $\Lambda=\Lambda_{B_2}^{\text{sc}}$.} As in \cref{egg:B2}, choose $\lambda_1$ and $\lambda_2$ such that 
		\[\alpha^\vee(\lambda_1)=1\quad \text{and}\quad\alpha^\vee(\lambda_2)=0\quad \text{and}\quad\beta^\vee(\lambda_1)=0\quad \text{and}\quad\beta^\vee(\lambda_2)=1.\]
		Then $s_\alpha(\lambda_1)=\lambda_1-\alpha$ and $s_\alpha(\lambda_2)=\lambda_2$ and $s_\beta(\lambda_1)=\lambda_1$ and $s_\beta(\lambda_2)=\lambda_2-\beta$. Thus,
		\[\Delta_\alpha(x_{\lambda_2})=\Delta_\beta(x_{\lambda_1})=0.\]
		In addition,
		\[\epsilon\Delta_\alpha(x_{\lambda_1})=\epsilon\left(\tfrac{x_{\lambda_1}-x_{\lambda_1-\alpha}}{x_\alpha}\right)=\epsilon\left(\tfrac{x_{\lambda_1}-(x_{\lambda_1}+_Fx_{-\alpha})}{x_\alpha}\right)=1.\]
		Similarly, $\epsilon\Delta_\beta(x_{\lambda_2})=1$. 
		Therefore, there exist $r_{I_1}^{I_2}\in R$ such that the ideal $\mcal{A}$ is generated by the relations
		\begin{enumerate}[label=(\alph*)]
			\item $0=\mrm{ev}_{x_{\lambda_1}}=\Xa+r_1^{\ab}\Xab+r_1^{\ba}\Xba+r_1^{\aba}\Xaba+r_1^{\bab}\Xbab+r_{1}^{\abab}\Xabab$,
			\item $0=\mrm{ev}_{x_{\lambda_2}}=\Xb+r_2^{\ab}\Xab+r_2^{\ba}\Xba+r_2^{\aba}\Xaba+r_2^{\bab}\Xbab+r_{2}^{\abab}\Xabab$.
		\end{enumerate}
		
		The rest of this calculation is performed by hand and is similar to the calculation in \cref{theo:A2}.
		\newline

		\underline{Case 2: $\Lambda=\Lambda_{B_2}^{\text{ad}}$.} The code in \cite{G} tells us that there are $r_{I_1}^{I_2}\in R$ such that the ideal $\mcal{A}$ is generated by the following relations:
		\begin{enumerate}[label=(\alph*)]
			\item \small{$0=\mrm{ev}_{x_\alpha}=2\Xa-\Xb-2a_{11}\Xab+a_{11}\Xba-(a_{11}^2+2a_{12})\Xbab+r_{\aaa}^{\abab}\Xabab$,}
			\item \small{$0=\mrm{ev}_{x_\beta}=-2\Xa+2\Xb+2a_{11}\Xab-3a_{11}\Xba-4(a_{11}^2+a_{12})\Xaba-a_{11}^2\Xbab+r_{\bbb}^{\abab}\Xabab$.}
		\end{enumerate}

		The rest of this calculation is performed by hand and is similar to the calculation in \cref{theo:A2}.
	\end{egg}
	
	\begin{rmk}
		The relations $\mcal{A}$ and $\mcal{M}$ for the group $G_2$ can be found in the output of \cite{G}. However, as there are a large number of relations in this case, we do not attempt to compute a minimal ring presentation for the ring $\mrm{h}^*(G_2)$ at this time.
	\end{rmk}

	
	\bibliographystyle{alphaurl}
	\bibliography{biblist}

\end{document}